\documentclass[11pt]{amsart}
\usepackage{amsfonts}
\usepackage{amsmath,amssymb,amsthm}
\usepackage{color,graphics}
\usepackage{bm}
\usepackage{amsmath}
\usepackage{amssymb}
\usepackage{graphicx}
\usepackage{amsmath,amsfonts,amssymb,amsthm,graphics,amscd,enumerate,verbatim}
\usepackage{color}
\usepackage[T1]{fontenc}

\newcommand{\CC}{\mathbb{C}}

\newcommand{\RR}{\mathbb{R}}

\def\Image{\mathop{\rm Im}}

\def\diag{{\rm diag\, }}

\newtheorem{theorem}{Theorem}

\newtheorem{proposition}{Proposition}[section]
\newtheorem{lemma}[proposition]{Lemma}

\theoremstyle{definition}
\newtheorem{remark}[proposition]{Remark}
\theoremstyle{plain}
\input{tcilatex}

\begin{document}
\title[Bi-monotone maps with respect to minus partial order]{Bi-monotone maps  on the set of all variance-covariance matrices with respect to minus partial order}
\author[G.~Dolinar]{Gregor Dolinar}
\address[Gregor Dolinar]{University of Ljubljana, Faculty of Electrical
Engineering, Tr\v{z}a\v{s}ka cesta 25, SI-1000 Ljubljana, Slovenia, and
IMFM, Jadranska 19, SI-1000 Ljubljana, Slovenia}
\email{gregor.dolinar@fe.uni-lj.si}
\author[D.~Ili\v{s}evi\'c]{Dijana~Ili\v{s}evi\'c}
\address[Dijana Ili\v{s}evi\'c]{University of Zagreb, Faculty of Science, Department of Mathematics, Bijeni\v cka 30, 10000 Zagreb, Croatia}
\email{ilisevic@math.hr}
\author[B.~Kuzma]{Bojan Kuzma}
\address[Bojan Kuzma]{University of Primorska, Glagolja\v{s}ka 8, SI-6000
Koper, Slovenia, and IMFM, Jadranska 19, SI-1000 Ljubljana, Slovenia}
\email{bojan.kuzma@upr.si}
\author[J.~Marovt]{Janko Marovt}
\address[Janko Marovt]{University of Maribor, Faculty of Economics and
Business, Razlagova 14, SI-2000 Maribor, Slovenia, and IMFM, Jadranska 19,
SI-1000 Ljubljana, Slovenia}
\email{janko.marovt@um.si}

\thanks{The first, the third, and the fourth author acknowledge the financial support from the Slovene Research Agency, ARRS (research programs
No.~P1-0285 and No.~P1-0288, and research projects No.~N1-0210 and  No. N1-0296). The authors also acknowledge the bilateral
project between Croatia and Slovenia (Generalized inverses with emphasis on searching for applications in statistics) which was financially supported by the Ministry of
Science and Education, Republic of Croatia, and by the Slovene Research Agency, ARRS (grant
BI-HR/20-21-036).}

\begin{abstract}
Let $H_{n}^{+}(\mathbb{R})$ be the cone of all positive semidefinite $n\times n$ real matrices. We describe the form of all surjective maps on $H_{n}^{+}(\mathbb{R}) $, $n\geq 3$, that preserve the minus partial order in
both directions.
\end{abstract}

\maketitle

\section{Introduction}

Let $M_{m,n}(\mathbb{F})$ be the set of all $m\times n$ matrices over a
field $\mathbb{F}$. Unless stated otherwise, $\mathbb{F} \in \{\mathbb{R},
\mathbb{C}\}$. If $m=n$, that is, for the set of all $n\times n$ square
matrices over $\mathbb{F}$ we will write $M_{n}(\mathbb{F)}$ instead of $M_{n,n}(\mathbb{F)}$. A matrix $A\in M_{n}(\mathbb{F)}$ has a (unique)
inverse $A^{-1}$ (a matrix satisfying $AA^{-1}=A^{-1}A=I$, where $I$ is the
identity matrix) if and only if it is a square matrix with linearly
independent columns (or rows). However, the need for some kind of a partial
inverse of a more general class of matrices (even non-square ones) has
arisen from applications. For given $A\in M_{m,n}(\mathbb{F})$, any matrix $X \in M_{n,m}(\mathbb{F})$ which is a solution to the equation $AXA=A$ is
called \textit{an inner generalized inverse} of $A$ and it is denoted by $A^{-}$. Every matrix $A\in M_{m,n}(\mathbb{F})$ has an inner generalized
inverse, but it is not unique in general. If $A \in M_n(\mathbb{F})$ has
inverse $A^{-1}$, then its inner generalized inverse is unique and equal to $A^{-1}$. This justifies the use of the word 'generalized'. More on inner
generalized inverses (and other generalized inverses, such as \textit{the Moore-Penrose inverse}  $A^{\dagger}$  for example) can be found in, e.g.,~\cite{B-I-T, MitraKnjiga}. As usual, let $A^{T}\in M_{n,m}(\mathbb{F})$
denote the transpose, $A^{\ast }\in M_{n,m}(\mathbb{C})$ the conjugate
transpose, $\text{Im} A$ the image (i.e.~the column space), and $\text{Ker} A
$ the kernel (the nullspace) of $A\in M_{m,n}(\mathbb{F})$.

Generalized inverses are closely related to partial orders on matrices.
There are many partial orders which may be defined on various sets of
matrices (e.g.,~the L\"owner, the minus, the sharp, the star, and also the
corresponding one-sided orders). \textit{The minus partial order} may be
defined on the full set $M_{m,n}(\mathbb{F})$. For $A,B\in M_{m,n}(\mathbb{F})$ we say that $A$ is below $B$ with respect to the minus partial order and
write
\begin{equation*}
A\leq ^{-}B\quad \text{when}\quad A^{-}A=A^{-}B\text{ and }AA^{-}=BA^{-}
\end{equation*}for some inner generalized inverse $A^{-}$ of $A$. The minus partial order
is also known as \textit{the rank subtractivity partial order} since for $A,B\in M_{m,n}(\mathbb{F})$ we have
\begin{equation}
A\leq ^{-}B\quad \text{if and only if}\quad \text{rank}(B-A)=\text{rank}(B)-\text{rank}(A).  \label{eq_rank_minus}
\end{equation}
Let us mention that the minus partial order was originally introduced by Hartwig in
\cite{Hartwig} and independently by Nambooripad in \cite{Nambooripad} on a
general regular semigroup however it was mostly studied on $M_{n}(\mathbb{F}) $ (see \cite{Mitra} and the references therein).

The minus partial order has some applications in statistics, especially in
the theory of linear statistical models (see \cite[Sections 15.3, 15.4]{MitraKnjiga}). Let
\begin{equation*}
y=X\beta +\epsilon
\end{equation*}be the matrix form of a linear model. We call $y$ \textit{the response vector} (also known as \textit{the observation vector}) and $X\in M_{m,n}(\mathbb{F})$ \textit{the model matrix }(also known as \textit{the design }or \textit{%
regressor matrix}). In order to complete the description of the model, some
assumptions about the nature of the errors have to be made. It is assumed
that $E(\epsilon )=0$ and $V(\epsilon )=\sigma ^{2}D$, i.e., the errors have
the zero mathematical expectation and covariances are known up to a scalar
(real number). Here $V$ denotes \textit{the variance-covariance matrix}
(also known as \textit{dispersion} or \textit{covariance matrix}). The
nonnegative parameter $\sigma ^{2}$ and the vector of parameters (real
numbers) $\beta \in \mathbb{R}^{n}$ are unspecified, and $D$ is a known $%
m\times m$ real, positive semidefinite matrix. We denote this linear model
with the triplet $(y,X\beta ,\sigma ^{2}D)$. Let us now present an example
of an application of the minus partial order in the theory of linear models
(see \cite[Theorem 15.3.6]{MitraKnjiga}).

\begin{proposition}
Let $L_{1}=(y,X_{1}\theta ,\sigma ^{2}D)$ and $L_{2}=(y,X_{2}\beta ,\sigma
^{2}D)$ be any two linear models. Then $X_{1}\leq ^{-}X_{2}$ if and only if
there exists a matrix $A$ with $\func{Im}A^{T}\subseteq \func{Im}X_{2}^{T}$
and $L_{1}$ is the model $L_{2}$ constrained by $A\beta =0$.
\end{proposition}

Given the model $M=(y,X\beta ,\sigma ^{2}I)$ one might rather work with the
transformed model $\hat{M}=(y,\hat{X}\beta ,\sigma ^{2}I)$ because the model
matrix $\hat{X}$ has more attractive properties then $X$ (e.g. elements of $%
X $ that are very close to zero are transformed to zero). It is natural to
demand that the transformed model still retains most of properties of the
original model (e.g. has similar relations to other transformed models) so
in view of the above application of the minus partial order it is
interesting to know what transformations $\Phi $ on $M_{m,n}(\mathbb{R})$,
or on $M_{n}(\mathbb{R})$, or on some subset of $M_{n}(\mathbb{R})$ preserve
the order $\leq ^{-}$ in one or in both directions.

A generalization of the concept of the minus partial order from matrices to
operators (more precisely, to the algebra $B(\mathcal{H})$ of all bounded
linear operators on a Hilbert space $\mathcal{H}$) was studied by \v{S}emrl
in \cite{Semrl}. Let $\mathcal{A}$ be some subset of $B(\mathcal{H})$. We
say that a map $\Phi :\mathcal{A}\rightarrow \mathcal{A}$ is a bi-monotone map with respect to
the minus partial order $\leq ^{-}$ (or that it preserves the minus
partial order $\leq ^{-}$ in both directions) when
\begin{equation*}
A\leq ^{-}B\quad \text{if and only if}\quad \Phi (A)\leq ^{-}\Phi (B)
\end{equation*}%
for every $A,B\in \mathcal{A}$. Linear bijective maps on $M_{n}(\mathbb{F})$%
, where $\mathbb{F}$ is a field of more then $n$ elements, that preserve the
minus partial order in one direction (i.e., $A\leq ^{-}B$ implies $\Phi
(A)\leq ^{-}\Phi (B)$) were investigated in \cite{Guterman1N}. Bijective bi-monotone maps with respect to the minus partial order on $M_{n}(\mathbb{F})$, $n\geq 3$,
were characterized in \cite{Legisa}, and those on $B(\mathcal{H})$, where a
Hilbert space $\mathcal{H}$ is infinite-dimensional, in \cite{Semrl}.

A set of matrices that deserves special attention is the set of all positive
semidefinite matrices. Such matrices are matrix analogues to nonnegative
real numbers. We may identify them with variance-covariance matrices since
(see, e.g.,~\cite{Christensen}) every variance-covariance matrix is positive
semidefinite and conversely, every (real) positive semidefinite matrix is a
variance-covariance matrix of some multivariate distribution. These matrices
appear not only in statistics but also as elements of the search space in
convex and semidefinite programming, as kernels in machine learning, as
conductivity matrices in thermodynamics, as diffusion tensors in medical
imaging, as overlap matrices in quantum chemistry, etc. More about the
fundamental role played by positive (semi)definite matrices can be found in
\cite{Bhatia}.

Let $H_{n}(\mathbb{F})$ denote the set of all Hermitian (i.e.,~symmetric in
the real case) matrices in $M_{n}(\mathbb{F})$ and let $H_{n}^{+}(\mathbb{F})
$ be the set of all positive semidefinite matrices in $H_{n}(\mathbb{F})$.
Note that $H_{n}^{+}(\mathbb{F})$ is not a linear space, but it is a convex
cone (it is closed under addition and scalar multiplication by non-negative
scalars only). Thus, linear preservers on $H_{n}^{+}(\mathbb{F})$ cannot be
studied, but at least additive ones can. Partial orders on $H_{n}^{+}(%
\mathbb{F})$ and the corresponding bi-monotone maps were studied in \cite{GolubicMarovt}.
One of the main results from \cite{GolubicMarovt} follows.

\begin{theorem}
Let $n\geq 3$ be an integer. Then $\Phi :H_{n}^{+}(\mathbb{R})\rightarrow
H_{n}^{+}(\mathbb{R})$ is a surjective, additive map that preserves the
minus order $\leq ^{-}$ in both directions if and only if there exists an
invertible matrix $S\in M_{n}(\mathbb{R)}$ such that
\begin{equation*}
\Phi (A)=SAS^{T}
\end{equation*}%
for every $A\in H_{n}^{+}(\mathbb{R})$.
\end{theorem}

In this paper we study bi-monotone maps with respect to the minus partial order on $%
H_{n}^{+}(\mathbb{R})$ without the additivity assumption.
In Section~2, we give some preliminary results. It turns out that ellipses centered at the origin in $\mathbb{R}^{2}$ can be
identified with operators in $H_{n}^{+}(\mathbb{R})$ whose images are equal
to a fixed 2-dimensional subspace of $\mathbb{R}^{n}$, while the points on these
ellipses correspond to rank-one positive semidefinite operators which are  dominated by the corresponding rank-two operators in $H_{n}^{+}(\mathbb{R})$ with respect to the minus partial order. Thus, in Section~3 we study incidence relation among concentric ellipses. The main result which
describes the form of all surjective bi-monotone maps with respect to the minus partial
order on $H_{n}^{+}(\mathbb{R}) $, $n\geq 3$, is stated and proved in
Section~4.

\section{Preliminaries}

Let us present some tools that will be useful throughout the paper. As
before, let $\mathbb{F=R}$ or $\mathbb{F=C}$. Let $P_{n}(\mathbb{F})$ be the
set of all projectors in $H_{n}^{+}(\mathbb{F})$, i.e., the set of all
orthogonal projection matrices in $M_{n}(\mathbb{F)}$. Let $V$ be a
subspace of $\mathbb{F}^{n}$. By $P_{V}\in P_{n}(\mathbb{F})$ we will denote
the projector with $\func{Im}P_{V}=V$. We will denote by $%
x\otimes y^{\ast }$ a rank one linear operator on $B(\mathcal{H})$ defined
by $(x\otimes y^{\ast })z=\left\langle z,y\right\rangle x$ for every $z\in
\mathcal{H}$. Note that every rank-one linear operator on $\mathcal{H}$ may
be written in this form. Recall that when $\mathcal{H}$ is finite
dimensional, say that the dimension of $\mathcal{H}$ equals $n$, we may
identify $B(\mathcal{H})$ with $M_{n}(\mathbb{F})$, and that in $\mathbb{F}%
^{n}$, $\left\langle z,y\right\rangle =y^{\ast }z$. It turns out that $P\in
P_{n}(\mathbb{F})$ is of rank-one if and only if $P=x\otimes x^{\ast }$ for
some $x\in \mathbb{F}^{n}$ with $\left\Vert x\right\Vert =1$. Indeed, let
first $x\in \mathbb{F}^{n}$ with $\left\Vert x\right\Vert =1$. Then $%
(x\otimes x^{\ast })^{\ast }=x\otimes x^{\ast }$ and for every $z\in \mathbb{%
F}^{n}$ we have $(x\otimes x^{\ast })^{2}z=\left\langle z,x\right\rangle
\left\langle x,x\right\rangle x=\left\langle z,x\right\rangle x=(x\otimes
x^{\ast })z$. Conversely, let $P$ be of rank-one and $P^{2}=P=P^{\ast }$.
Then $P=x\otimes y^{\ast }$ for some nonzero $x,y\in \mathbb{F}^{n}$. By
transferring the appropriate scalar to the second factor we may assume
without loss of generality that $\left\Vert x\right\Vert =1$. From $x\otimes
y^{\ast }=(x\otimes y^{\ast })^{\ast }=y\otimes x^{\ast }$ it follows that $%
y=\mu x$ for some nonzero $\mu \in \mathbb{F}$ and $(x\otimes y^{\ast
})^{2}=x\otimes y^{\ast }$ then implies that $\mu =1$, i.e., $P=x\otimes
x^{\ast }$ with $\left\Vert x\right\Vert =1$. The following auxiliary result
was proved in \cite{Semrl}.

\begin{lemma}
\label{Lemma_Semrl}Let $\mathcal{H}$ be a Hilbert space and let $x,y\in
\mathcal{H}$ be nonzero vectors and $A\in B(\mathcal{H})$. Then $x\otimes
y^{\ast }\leq ^{-}A$ if and only if $x\in \func{Im}A$ and there exists $z\in
H$ such that $\left\langle x,z\right\rangle =1$ and $y=A^{\ast }z$. In
particular, rank-one $x\otimes x^{\ast }\in H_{n}^{+}(\mathbb{F})$ satisfies
$x\otimes x^{\ast }\leq ^{-}A$ if and only if $x\in \func{Im}A$ and there
exists $z\in \mathbb{F}^{n}$ such that $\left\langle x,z\right\rangle =1$
and $x=Az$.
\end{lemma}

We denote by $e_{1},e_{2},\ldots ,e_{n}\in \mathbb{F}^{n}$ the standard
basis vectors. In the next lemma we answer the question of  when a rank-one positive
semidefinite matrix is dominated with respect to the minus partial order by
a given positive semidefinite matrix.

\begin{lemma} \label{lem:2.2}
\label{Lemma_ellipse}Let $A\in H_{n}^{+}(\mathbb{F})$ be nonzero with the
eigenvalue decomposition
\begin{equation*}
A=\alpha _{1}e_{i_{1}}\otimes e_{i_{1}}^{\ast }+\alpha _{2}e_{i_{2}}\otimes
e_{i_{2}}^{\ast }+\ldots +\alpha _{j}e_{i_{j}}\otimes e_{i_{j}}^{\ast }
\end{equation*}%
where  $ i_1, i_2, \dots, i_j \in \{1,2,\dots,n\}$ are pairwise distinct, $ 1 \leq j \leq n $, and
$\alpha _{k}>0$ for every $k\in \left\{
1,2,\ldots ,j\right\} $. Let $x\in \mathbb{F}^{n}$ be a nonzero vector. Then
$x\otimes x^{\ast }\leq ^{-}A$ if and only if there exist $\beta _{1},\beta
_{2},\ldots ,\beta _{j}\in \mathbb{F}\ $such that%
\begin{equation}
x=\beta _{1}e_{i_{1}}+\beta _{2}e_{i_{2}}+\ldots +\beta _{j}e_{i_{j}}\quad
\text{and\quad }\frac{\left\vert \beta _{1}\right\vert ^{2}}{\alpha _{1}}+%
\frac{\left\vert \beta _{2}\right\vert ^{2}}{\alpha _{2}}+\ldots +\frac{%
\left\vert \beta _{j}\right\vert ^{2}}{\alpha _{j}}=1.
\label{en_Lemma_ellipse}
\end{equation}
\end{lemma}

\begin{proof}
Suppose first that $x\otimes x^{\ast }\leq ^{-}A$. By Lemma %
\ref{Lemma_Semrl}, $x\in \func{Im}A$ and thus there exist $\beta _{1},\beta
_{2},\ldots ,\beta _{j}\in \mathbb{F}$ such that $x=\beta
_{1}e_{i_{1}}+\beta _{2}e_{i_{2}}+\ldots +\beta _{j}e_{i_{j}}$. Also (see
Lemma \ref{Lemma_Semrl}), there exists $z\in \mathbb{F}$, such that $%
\left\langle x,z\right\rangle =1$ and $x=Az$. It follows that $\left\langle
\beta _{1}e_{i_{1}}+\beta _{2}e_{i_{2}}+\ldots +\beta
_{j}e_{i_{j}},z\right\rangle =1$ and therefore%
\begin{equation*}
\beta _{1}\left\langle e_{i_{1}},z\right\rangle +\beta _{2}\left\langle
e_{i_{2}},z\right\rangle +\ldots +\beta _{j}\left\langle
e_{i_{j}},z\right\rangle =1.
\end{equation*}%
From $Az=x$ we have $\alpha _{1}\left\langle z,e_{i_{1}}\right\rangle
e_{i_{1}}+\alpha _{2}\left\langle z,e_{i_{2}}\right\rangle e_{i_{2}}+\ldots
+\alpha _{j}\left\langle z,e_{i_{j}}\right\rangle e_{i_{j}}=$ $\beta
_{1}e_{i_{1}}+\beta _{2}e_{i_{2}}+\ldots +\beta _{j}e_{i_{j}}$ and thus $%
\alpha _{k}\left\langle z,e_{i_{k}}\right\rangle =\beta _{k}$ for every $%
k\in \left\{ 1,2,\ldots ,j\right\} $. So,
\begin{equation*}
\left\langle e_{i_{k}},z\right\rangle =\frac{\overline{\beta _{k}}}{\alpha
_{k}},\quad k=1,2,\ldots ,j.
\end{equation*}%
It follows that $\beta _{1}\frac{\overline{\beta _{1}}}{\alpha _{1}}+\beta
_{2}\frac{\overline{\beta _{2}}}{\alpha _{2}}+\ldots +\beta _{j}\frac{%
\overline{\beta _{j}}}{\alpha _{j}}=1$ and therefore%
\begin{equation*}
\frac{\left\vert \beta _{1}\right\vert ^{2}}{\alpha _{1}}+\frac{\left\vert
\beta _{2}\right\vert ^{2}}{\alpha _{2}}+\ldots +\frac{\left\vert \beta
_{j}\right\vert ^{2}}{\alpha _{j}}=1.
\end{equation*}

Conversely, let $\beta _{1},\beta _{2},\ldots ,\beta _{j}\in \mathbb{F}\ $be
such that $x=\beta _{1}e_{i_{1}}+\beta _{2}e_{i_{2}}+\ldots +\beta
_{j}e_{i_{j}}$ and
\begin{equation*}
\frac{\left\vert \beta _{1}\right\vert ^{2}}{\alpha _{1}}+\frac{\left\vert
\beta _{2}\right\vert ^{2}}{\alpha _{2}}+\ldots +\frac{\left\vert \beta
_{j}\right\vert ^{2}}{\alpha _{j}}=1.
\end{equation*}%
Let $z~=~\left( z_{1},z_{2},\ldots ,z_{n}\right) ^{T}\in \mathbb{F}^{n}$ be
such that $z_{i_{k}}=\frac{\beta _{k}}{\alpha _{k}}$ for every $k\in \left\{
1,2,\ldots ,j\right\} $ while other (possible) components can be arbitrary.
Then
\begin{equation*}
\left\langle x,z\right\rangle =\frac{\left\vert \beta _{1}\right\vert ^{2}}{%
\alpha _{1}}+\frac{\left\vert \beta _{2}\right\vert ^{2}}{\alpha _{2}}%
+\ldots +\frac{\left\vert \beta _{j}\right\vert ^{2}}{\alpha _{j}}=1
\end{equation*}%
and
\begin{equation*}
Az=\alpha _{1}\left\langle z,e_{i_{1}}\right\rangle e_{i_{1}}+\alpha
_{2}\left\langle z,e_{i_{2}}\right\rangle e_{i_{2}}+\ldots +\alpha
_{j}\left\langle z,e_{i_{j}}\right\rangle e_{i_{j}}=\alpha _{1}\frac{\beta
_{1}}{\alpha _{1}}e_{i_{1}}+\alpha _{2}\frac{\beta _{2}}{\alpha _{2}}%
e_{i_{2}}+\ldots +\alpha _{j}\frac{\beta _{j}}{\alpha _{j}}e_{i_{j}}=x
\end{equation*}%
and thus by Lemma \ref{Lemma_Semrl}, $x\otimes x^{\ast }\leq ^{-}A$.
\end{proof}

\begin{remark}
\label{rem:geometric} Notice that equation (\ref{en_Lemma_ellipse}) is
equivalent to $x^{\ast }A^{\dagger }x=1$ where $A^{\dagger }$ is the
Moore-Penrose inverse of the matrix $A$. Therefore, every rank-one operator $%
x\otimes x^{\ast }\in H_{n}^{+}(\mathbb{F})$ which is majorized by a nonzero
operator $A\in H_{n}^{+}(\mathbb{F})$ is induced by a vector $x\in \func{Im}%
A $ which determines a point that lies on an ellipsoid $\left\{ x\in \func{Im%
}A:x^{\ast }A^{\dagger }x=1\right\} $. Clearly, every nonzero $x=Az\in \func{%
Im}A$ satisfies $x^{\ast }A^{\dagger }x=z^{\ast }AA^{\dagger }Az=z^{\ast
}Az=z^{\ast }\sqrt{A}\sqrt{A}z=\Vert \sqrt{A}z\Vert ^{2}\neq 0$ so, when
appropriately scaled belongs to ellipsoid induced by $A$. It follows that if
$B\in H_{n}^{+}(\mathbb{F})\backslash \{0\}$ induces the same ellipsoid $%
\mathfrak{E}$ as $A\in H_{n}^{+}(\mathbb{F})$, then each nonzero $x\in \func{%
Im}A$, when appropriately scaled, lies on $\mathfrak{E}$ so satisfies $x\in
\func{Im}B$. Thus, $\func{Im}A=\func{Im}B$. Moreover, $x^{\ast }A^{\dagger
}x=1$ if and only if $x^{\ast }B^{\dagger }x=1$ for each $x\in \func{Im}A=%
\func{Im}B$. Next, $AA^\dagger$ is a projector onto $\func{Im}A$, so   $A^{\dagger }x=(A^{\dagger }AA^{\dagger })x=A^{\dagger }(AA^{\dagger })x=0$ for each  $x\in \left( \func{Im}A\right) ^{\bot }$, and likewise $B^{\dagger }x=0$ for
each  $x\in \left( \func{Im}B\right) ^{\bot }=\left( \func{Im}A\right) ^{\bot }$. Therefore, $x^{\ast }\left(
A^{\dagger }-B^{\dagger }\right) x=0$ for every $x\in \mathbb{F}^{n}$, and
since $A^{\dagger }$ and $B^{\dagger }$ are Hermitian matrices, $A=B$.
This shows that $A = B$ is equivalent to $C \leq ^{-} A$ if
and only if $C \leq ^{-} B$ for every rank-one $C \in H_{n}^{+}(\mathbb{F})$.
\end{remark}

In particular, ellipses centered at the origin in $\mathbb{R}^{2}$ can be
identified with operators in $H_{n}^{+}(\mathbb{R})$ whose images are equal
to a fixed 2-dimensional subspace of $\mathbb{R}^{n}$ while points on these
ellipses correspond to rank-one positive semidefinite operators majorized by
the corresponding rank-two operators in $H_{n}^{+}(\mathbb{R})$.

Given a rank-two $A\in H_{n}^{+}(\mathbb{R})$, let
\begin{equation} \label{eq:ellipse}
\mathfrak{E}_{A}=\left\{ x\in \func{Im}A:x^{\ast }A^{\dagger }x=1\right\}
\end{equation}%
be the corresponding ellipse induced by $A$.

\section{Incidence relation among concentric ellipses}
\begin{lemma}\label{lem:3.1}
 Let $\Pi_1,\Pi_2\subseteq\RR^n$ be two two-dimensional subspaces, and let $E_1\subseteq\Pi_1$ and $E_2\subseteq\Pi_2$ be the ellipses, centered at the origin. The following are equivalent.
 \begin{itemize}
  \item[(i)] Ellipses $E_1$ and $E_2$ lie in the same space (i.e. $\Pi_1=\Pi_2$).
  \item[(ii)] There exists an ellipse $E$, which may coincide with $E_1$ or with  $E_2$,  centered at the origin  and lying is some two-dimensional subspace $\Pi$ (perhaps $E=E_1$) such that $\#(E\cap E_1)=\#(E\cap E_2)\ge 4$.
 \end{itemize}
\end{lemma}
\begin{proof}
If $\#(E_1\cap E)\ge 4$, then the two-dimensional spaces $\Pi_1$ and $\Pi$ contain at least four nonzero points which cannot belong to the same line. Hence, $\Pi=\Pi_1$.

Conversely, if $E_1,E_2$ belong to the same plane then it is easy to construct another ellipse $E$ with the required properties.
\end{proof}

Throughout the rest of this section the ambient space is the plane $\RR^2$. We assume that all ellipses are concentric (i.e., their axes intersect at the same point, called the center of ellipse) with the center at  the origin.
Thus, they are given by the equation $x^T Qx=1$ and denoted by $E_Q$, where $Q \in H_{2}^{+}(\mathbb{R})$ is a positive-definite matrix.
Note that an ellipse $E_Q$ lies inside the unit circle if and only if both eigenvalues of $Q$ are greater  or equal to $1$.
To shorten the narrative, a concentric ellipse whose axes lie on the coordinate axes will be called a \emph{standard ellipse}.

Let $K_{\rho}$ be a circle with the center at the origin and radius $\rho$. In particular, ${\bf K} = K_1$ is the unit circle.
We will consider bijection $\Phi\colon\RR^2\to\RR^2$ which map concentric ellipses into concentric ellipses. Hence, if $E$ is an ellipse, then $\Phi(E)$ is either an ellipse or a part of a uniquely determined ellipse (uniqueness follows because $\Phi$ is injective, so the set $\Phi(E)$ is infinite). We denote by $\mathrm{int}\Phi(E)$ the interior of this unique ellipse.

Let us start with the following.

\begin{lemma}\label{lem:preserves-containment}
 Let $\Phi\colon\RR^2\to\RR^2$ be a bijection which maps concentric ellipses into concentric ellipses. Then, $\Phi$ uniformly preserves containment of ellipses inside each other or uniformly reverses it.
\end{lemma}
\begin{proof}
Let $E_1,E_2,E_3$ be concentric ellipses.
It is easy to see the equivalence of the following statements:
\begin{itemize}
  \item[(i)] $E_1\subseteq\mathrm{int}E_2\subseteq \mathrm{int} E_3$ or  $E_1\supseteq\mathrm{int}E_2\supseteq \mathrm{int} E_3$.
  \item[(ii)]
  \begin{itemize}
    \item[(a)]   $E_1,E_2,E_3$ are pairwise disjoint.
    \item[(b)] If some ellipse $E$ touches  $E_1$ and $E_3$, then $E$ intersects $E_2$ in at least four points.
  \end{itemize}
  \end{itemize}
  Suppose $E_1,E_2,E_3$ satisfy the assumptions of (i). Then one can always find $E$ which touches $E_1$ and $E_3$, and hence intersects $E_2$ in at least four points. Applying the above equivalence to $\Phi$-images and using injectiveness of $\Phi$  we see that either $\Phi(E_1)\subseteq\mathrm{int}\Phi(E_2)\subseteq\mathrm{int}\Phi(E_3)$ or else $\mathrm{int}\Phi(E_1)\supseteq\mathrm{int}\Phi(E_2)\supseteq\Phi(E_3)$.  Arguing recursively, the same conclusion holds for any finite collection  of nested ellipses. Since each interior point of $E_3$ belongs to some concentric ellipse contained in $\mathrm{int}E_3$ we deduce that either $\Phi$ maps $\mathrm{int}E_3$ into the interior of the ellipse $\Phi(E_3)$ or into its exterior.

  We claim that $\Phi$ does this uniformly for each ellipse. Assume otherwise, and let $\hat{E}_3$ be a concentric ellipse such that $\Phi$ maps $\mathrm{int}E_3$ into interior of $\Phi(E_3)$ and  $\mathrm{int}\hat{E}_3$ into exterior of $\Phi(\hat{E}_3)$. If $E_3$ and $\hat{E}_3$ do not intersect, then one is contained in another, say $E_3\subseteq\mathrm{int}\hat{E}_3$.  We may enlarge this inclusion by adding  additional ellipse $E_2$ contained in $\mathrm{int}E_3$; then on the one hand, $\Phi(E_2)\subseteq \mathrm{int}\Phi(E_3)$, so by the above, $\Phi(E_2)\subseteq \mathrm{int}\Phi(E_3)\subseteq\mathrm{int}\Phi(\hat{E}_3)$ but on the other hand, $\Phi(\hat{E}_3)\subseteq \mathrm{int}\Phi(E_3)$, so again, by the previous arguments,  $\Phi(\hat{E}_3)\subseteq \mathrm{int}\Phi(E_3)\subseteq \mathrm{int}\Phi(E_2)$, which implies that $\Phi(E_2)\subseteq \mathrm{int}\Phi(E_2)$ a contradiction.

  Assume $E_3$ and $\hat{E}_3$ intersect or touch. Then we choose an ellipse $E_4$ which contains both $E_3$ and $\hat{E}_3$ in its interior and we choose ellipse $E_2\subseteq \mathrm{int} E_3\cap \mathrm{int}\hat{E}_3$. The inclusion $E_2\subseteq \mathrm{int} E_3\subseteq \mathrm{int} E_4$ is preserved, while the inclusion $E_2\subseteq \mathrm{int} \hat{E}_3\subseteq \mathrm{int} E_4$ is reversed, so $\Phi(E_2)$ belongs to interior and to exterior of ellipse $\Phi(E_4)$, a contradiction.
\end{proof}

\begin{lemma}\label{lem:Phi(interior)=interior} Let $\Phi\colon\RR^2\to\RR^2$ be a bijection which maps concentric ellipses into concentric ellipses and its restriction to the unit circle ${\bf K}$ equals  $x\mapsto Tx/\|Tx\|$ for some invertible diagonal matrix~$T=\diag(\mu,\lambda)$.  Then, $\Phi$ fixes ${\bf K}$ set-wise and fixes the points $(\pm1,0)$ and $(0,\pm1)$. Moreover, it  leaves  the interior of the unit circle invariant.
\end{lemma}
\begin{proof} The first part is clear. For the second part, the previous lemma implies that $\Phi$ either maps~$\mathrm{int}{\bf K}$ into itself or into exterior of the unit circle.
 Assume the former.
 Then the previous lemma implies that $\Phi$ also maps   its exterior onto its interior.
Let us choose a  circle  $K_\rho$ of radius $\rho<1$ and let $E_1$ and $E_2$ be  two standard ellipses,
touching  the unit circle ${\bf K}$ at  $(\pm 1,0)$ and $(0,\pm 1)$ respectively and intersecting  at    $\frac{1}{\sqrt{2}}(\pm \rho, \pm \rho)$; note that $E_1$ and $E_2$ intersect on the circle $K_\rho$.

By the assumptions, the ellipses $\Phi(E_1)$ and $\Phi(E_2)$ lie outside the unit circle and touch it at $(\pm 1,0)$ and $(0,\pm 1)$ respectively
because these four points  are fixed by $\Phi$.
Thus, $\Phi(E_1)$ and $\Phi(E_2)$ are also standard ellipses.
We note that  $\Phi(E_1)$ lies in the vertical strip $[-1,1] \times \mathbb{R}$ and $\Phi(E_2)$ lies in the horizontal strip $\mathbb{R} \times [-1,1]$.
Therefore, $\Phi(E_1)$ and $\Phi(E_2)$ intersect outside the unit circle but inside the square spanned by $(\pm 1, 0)$ and $(0, \pm 1)$.
Since $\Phi(E_1)$ and $\Phi(E_2)$ are standard ellipses, they   intersect  at points that are placed symmetrically relatively to the coordinate axes.
\begin{figure}[h!]
 \centering
   \includegraphics[width=2cm]{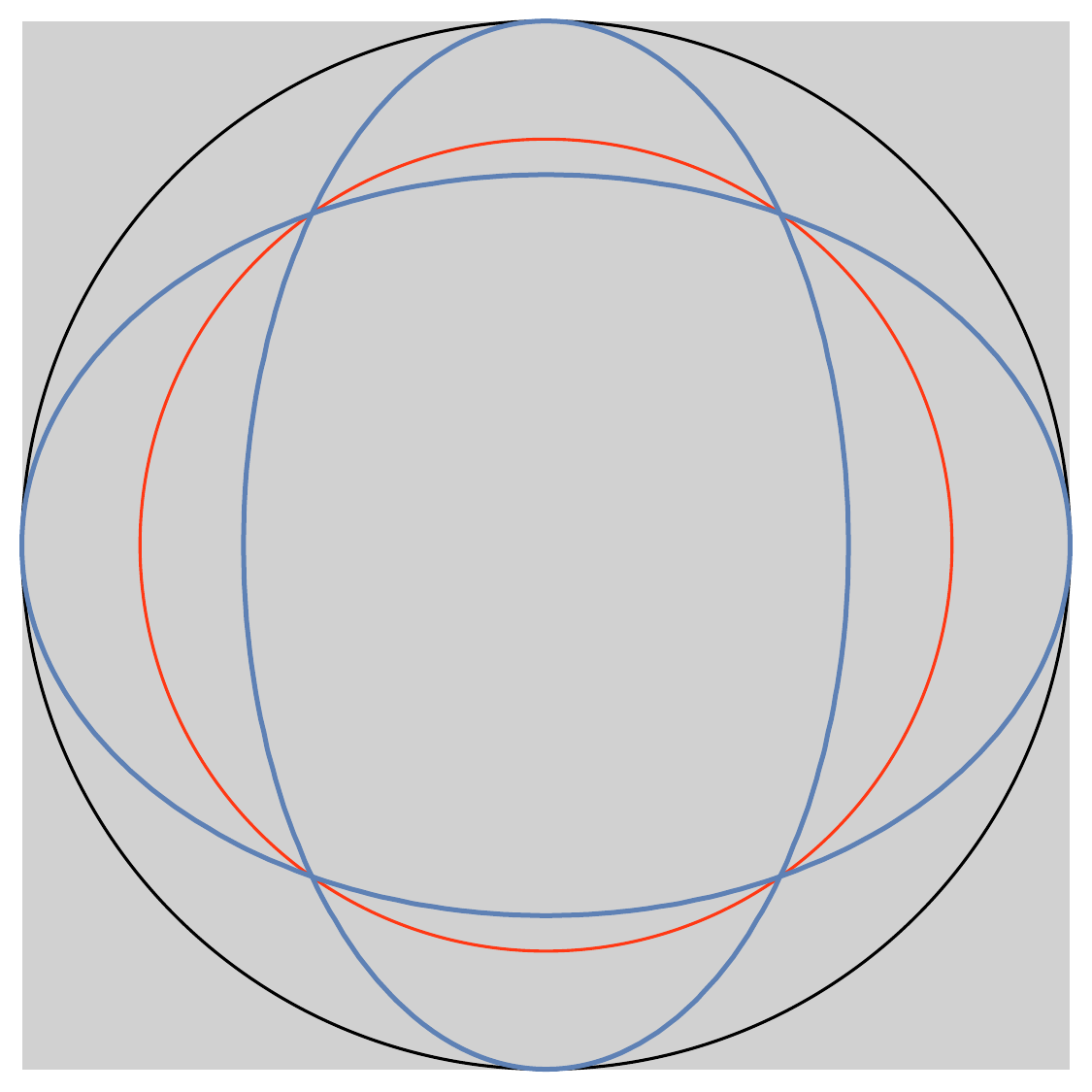}\hspace{2cm}\includegraphics[width=4cm]{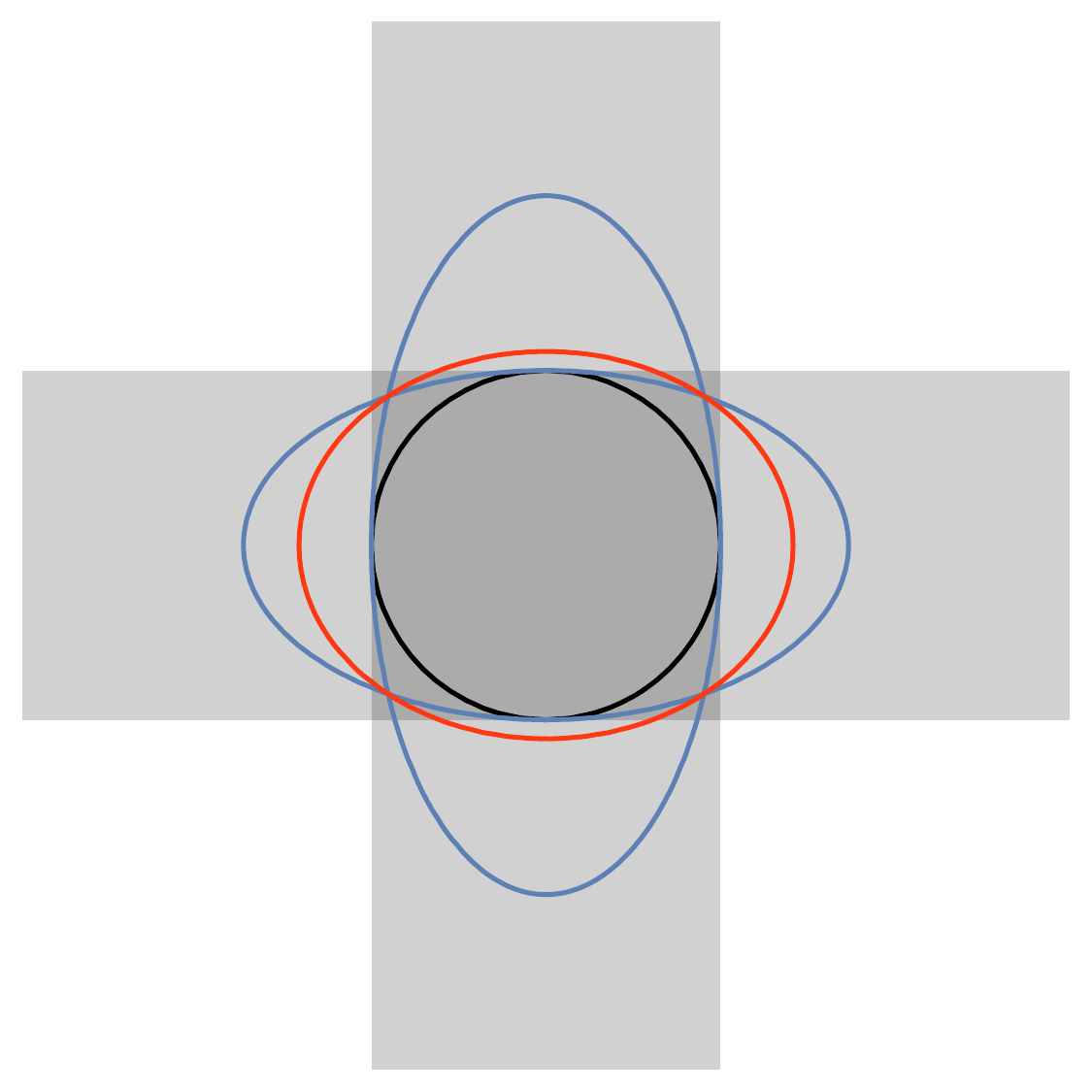}
   \caption{$E_1,E_2$ touch the unit circle and intersect at a concentric circle $K_\rho$. The image of $K_\rho$ is an ellipse passing through the symmetrically placed points of intersection of $\Phi(E_1)$ and $\Phi(E_2)$. These points belong to the interior of the shaded rectangle $[-1,1]\times[-1,1]$.}
 \end{figure}

The ellipses $E_1$ and $E_2$ intersect on $K_\rho$, so $\Phi(E_1)$ and $\Phi(E_2)$ intersect on the ellipse $\Phi(K_\rho)$.
Thus, the ellipse $\Phi(K_\rho)$ contains four symmetric points inside the square $[-1,1] \times [-1,1]$ and therefore  its axes lie on the coordinate axes.

On the other hand, let us choose ellipses $E_1^{'}$ and $E_2^{'}$, again with axes lying on the coordinate axes and touching the unit circle at $(\pm 1, 0)$ and $(0, \pm 1)$ respectively, but touching the circle $K_\rho$ at $(0, \pm \rho)$ and $(\pm \rho, 0)$ respectively.
Then $\Phi(E_1^{'})$ and $\Phi(E_2^{'})$ are ellipses touching the unit circle and lying in its exterior, so one of their axes is 1 and the other ($b_\rho$, respectively~$a_\rho$) is greater than $1$. Since $\Phi(K_\rho)$ lies outside the unit circle, its axes lie on the coordinate axes and it touches both $\Phi(E_1^{'})$ and $\Phi(E_2^{'})$, its equation is $$\frac{x^2}{a_\rho^2} + \frac{y^2}{b_\rho^2}=1.$$
Furthermore, since $\Phi(K_\rho)$ contains four symmetric points inside the square $[-1,1] \times [-1,1]$, the vertices of that square, i.e., the points $(\pm 1, \pm 1)$ lie outside $\Phi(K_\rho)$.
Thus, $$\frac{1}{a_\rho^2} + \frac{1}{b_\rho^2} > 1.$$
Then at least one of $a_\rho$, $b_\rho$ is less than $\sqrt{2}$ and this is true for every radius $\rho<1$.

If for $\Phi(K_{\rho_1})$ we have $b_{\rho_1} \geq \sqrt{2}$ and for $\Phi(K_{\rho_2})$ we have $a_{\rho_2} \geq \sqrt{2}$ then $a_{\rho_1} < \sqrt{2}$ and $b_{\rho_2} < \sqrt{2}$, so  $\Phi(K_{\rho_1})$ and $\Phi(K_{\rho_2})$ intersect; a contradiction.
Therefore, either for every radius $\rho<1$ we have $a_\rho < \sqrt{2}$, or for every radius $\rho<1$ we have $b_\rho < \sqrt{2}$.
In the first case $\Phi(K_\rho)$ lies in the vertical strip $(-\sqrt{2}, \sqrt{2}) \times \mathbb{R}$ for each concentric circle $K_\rho$ contained inside the unit circle.
Each point in the interior of ${\bf K}$ lies on some concentric circle and hence it is mapped into that vertical strip.
 Then, in particular, the point $(4,0)$ cannot be the image of any point inside the unit circle.
It  cannot be the image of any point outside the unit circle either since $\Phi$ maps exterior of the unit circle onto its interior.
The case when $b_\rho < \sqrt{2}$ for every radius $\rho<1$ analogously contradicts surjectivity of $\Phi$.
\end{proof}

We continue with a technical lemma which will help us
in calculating the  touching ellipses.

Let $R_{\phi }=\left(
\begin{smallmatrix}
 \cos \phi  & -\sin \phi  \\
 \sin \phi  & \cos \phi  \\
\end{smallmatrix}\right)$ be a rotation matrix.
Observe that  the angle between
one of the axis of the ellipse induced by the matrix
$$ R_{\phi }\left( \begin{smallmatrix}
 1 & 0 \\
 0 & r \\
\end{smallmatrix}
\right)
R_{\phi }^T $$
and the $x$-axis is equal to $\phi$. We say that such ellipse is at angle $\phi$ relative to $x$-axis.

\begin{lemma}\label{Lemma_first_existence}
Let an ellipse $E_Q$, induced by a matrix $Q$, lie in the interior of the unit circle and let $\phi \in (0, \frac{\pi}{2})$ be a given angle.
	Then there exists a unique ellipse $E$, at angle $\phi$ relative to $x$-axis, simultaneously touching the unit circle and the ellipse $E_Q$. It is induced by the matrix
	$ R_{\phi }\left( \begin{smallmatrix}
	1 & 0 \\
	0 & r \\
	\end{smallmatrix}
	\right)
	R_{\phi }^T $, where $r$ satisfies the equation
	\begin{equation} \label{eq:ellipse-equation}
	 \det\left( Q-  R_{\phi }
	\left(\begin{smallmatrix}
	1 & 0 \\
	0 & r \\
	\end{smallmatrix}\right)
	R_{\phi }^T\right) =0.
	\end{equation}
\end{lemma}

\begin{proof}
The ellipse $E$ touches the unit circle, hence one of its  axis has length $1$, and the other has length $\rho$ which is yet to be determined. Since the ellipse
$E$ is at an angle $\phi$ relative to $x$-axis, $E$ is obviously induced by
$ R_{\phi }\left( \begin{smallmatrix}
 1 & 0 \\
 0 & r \\
\end{smallmatrix}
\right)
R_{\phi }^T $, where $r = 1/\rho^2$.
We need to determine $r$ so that $E_Q$ and $E$  touch at some vector $x$. Then,
 \begin{equation}\label{eq:qf1}
 x^T Qx=1= x^T R_{\phi } \left(
\begin{smallmatrix}
 1 & 0 \\
 0 & r \\
\end{smallmatrix}
\right)  R_{\phi }^Tx
\end{equation}
which implies
$$x^T \left(Q- R_{\phi }\left(
\begin{smallmatrix}
 1 & 0 \\
 0 & r \\
\end{smallmatrix}
\right) R_{\phi }^T \right)x=0.$$
Moreover, gradients of these two ellipses at the point where they touch are parallel.
Now, the gradient of a quadratic form $x\mapsto  x^TAx$ is given by
$$\partial  x^TAx=x^T (A^T+A ) $$
and since in our case the matrices $Q$ and $R_{\phi }\left(
\begin{smallmatrix}
	1 & 0 \\
	0 & r \\
\end{smallmatrix}
\right) R_{\phi }^T$ are symmetric,  we get that for some scalar $\lambda$, $$x^TQ=\lambda x^T  R_{\phi }\left(
\begin{array}{cc}
 1 & 0 \\
 0 & r \\
\end{array}
\right)  R_{\phi }^T.$$
Multiplying this by $x$ and using \eqref{eq:qf1}
we obtain $\lambda=1$, so
$x^T (Q-R_{\phi }\left(
\begin{smallmatrix}
 1 & 0 \\
 0 & r \\
\end{smallmatrix}
\right)  R_{\phi }^T)=0$. Thus,  $x$ is an eigenvector of the symmetric matrix  $Q-  R_{\phi } \left(
\begin{smallmatrix}
 1 & 0 \\
 0 & r \\
\end{smallmatrix}
\right)R_{\phi }^T$  corresponding to the eigenvalue $0$.
Therefore, we are looking for $r$ such that
$$\det\left( Q-R_{\phi }
\left(\begin{smallmatrix}
 1 & 0 \\
 0 & r \\
\end{smallmatrix}\right)
 R_{\phi }^T\right) =0.$$
	This is equivalent to
	\begin{equation}\label{st}
\det(R_\phi^T QR_{\phi} -\left(\begin{smallmatrix}
1 & 0 \\
0 & r \\
\end{smallmatrix}\right))=0.
	\end{equation}
	Since the ellipse $E_Q$ is inside the unit circle,  $x^T Qx>1$ for each unit vector $x$, so the $(1,1)$ entry of $R_\phi^T QR_{\phi} -\left(\begin{smallmatrix}
	1 & 0 \\
	0 & r \\
	\end{smallmatrix}\right)$ is nonzero. Thus, the left hand side of \eqref{st} is a linear polynomial in~$r$ and hence there does exist a unique $r$ satisfying \eqref{st}.
	
	Moreover, with such $r$, the matrix $ Q-R_{\phi }\left(\begin{smallmatrix}
	1 & 0 \\
	0 & r \\
	\end{smallmatrix}\right)
	R_{\phi }^T$ is singular but nonzero, and hence it has exactly one nonzero eigenvalue. Therefore, it is a semi-definite $2$-by-$2$ matrix,  and  it follows easily after its diagonalization that, modulo a scalar multiple, there exists at most one nonzero $x$ with $$x^T(Q-R_{\phi }\left(\begin{smallmatrix}
 1 & 0 \\
 0 & r \\
\end{smallmatrix}\right)
  R_{\phi }^T)x=0.$$
  It remains to prove that at least one such $x$ does exist. In fact, if $x$ is a unit vector in the  kernel of $Q-R_{\phi }\left(\begin{smallmatrix}
 1 & 0 \\
 0 & r \\
\end{smallmatrix}\right)
  R_{\phi }^T$ , then $Qx=R_{\phi }\left(\begin{smallmatrix}
 1 & 0 \\
 0 & r \\
\end{smallmatrix}\right)
  R_{\phi }^Tx$. We only need to scale $x$ to conclude that it 
  belongs to ellipses induced by $Q$ and induced by $R_{\phi }\left(\begin{smallmatrix}
 1 & 0 \\
 0 & r \\
\end{smallmatrix}\right)
  R_{\phi }^T$.
	
\end{proof}

\begin{remark}As a computational complement to Lemma \ref{Lemma_first_existence} let us note that
\begin{equation}\label{Rmatrix}
R_{\phi }
\left(\begin{matrix}
1 & 0 \\
0 & r \\
\end{matrix}\right)
R_{\phi }^T=
\left(\begin{matrix}
(r-1)\sin^2{\phi}+1 & -(r-1)\sin{\phi}\cos{\phi} \\
-(r-1)\sin{\phi}\cos{\phi} & (r-1)\cos^2{\phi}+1 \\
\end{matrix}\right).
\end{equation}
If $Q$ in Lemma \ref{Lemma_first_existence} is given by
$Q=\left(\begin{smallmatrix}
a & 0 \\
0 & b \\
\end{smallmatrix}\right)$
then
$$Q-R_{\phi }
\left(\begin{matrix}
1 & 0 \\
0 & r \\
\end{matrix}\right)
R_{\phi }^T = \left(\begin{matrix}
a-r\sin^2{\phi}-\cos^2{\phi} & (r-1)\sin{\phi}\cos{\phi} \\
(r-1)\sin{\phi}\cos{\phi} & b-\sin^2{\phi}-r\cos^2{\phi} \\
\end{matrix}\right),$$
so \eqref{eq:ellipse-equation} first implies
$$r-1=\frac{(a-1)(b-1)}{(a-1)\cos^2{\phi}+(b-1)\sin^2{\phi}}$$
and then
$$r=\frac{(a-1)b+a(b-1)\tan^2{\phi}}{(a-1)+(b-1)\tan^2{\phi}}.$$
For such $r$, the matrix \eqref{Rmatrix} becomes
\begin{eqnarray}\label{Rmatrixwithr}
\frac{1}{(a-1)+(b-1)\tan^2{\phi}}
\left(\begin{matrix}
	(a-1)+a(b-1)\tan^2{\phi} & -(a-1)(b-1)\tan{\phi} \\
	-(a-1)(b-1)\tan{\phi}  & (a-1)b+(b-1)\tan^2{\phi} \\
\end{matrix}\right).
\end{eqnarray}
\end{remark}

\medskip

From now on, $\Phi\colon\RR^2\to\RR^2$ is a surjection, which maps concentric ellipses into concentric ellipses, maps interior of ${\bf K}$ onto itself, and maps a point $(c,s)\in\RR^2$ on the unit circle  to the point
 $$\frac{(\mu c, \lambda s)}{\|(\mu c, \lambda s) \|}$$
again  on the unit circle,
 i.e., if $x=(c,s)\in{\bf K}$, then $\Phi(x)=Tx/\|Tx\|$ where $T=\diag(\mu,\lambda)$ is a diagonal matrix.
 Without loss of generality, due to polar decomposition, we may assume  $\lambda$ and $\mu$ are both positive.
 Within a sequence of steps we will show that $\lambda=\mu$ and that $\Phi$ fixes all the points from $\RR^2$. \bigskip

{Step 1}. The points $(\pm 1,0)$ and $(0,\pm 1)$ are fixed.

\noindent Proof. Obvious.\medskip

{Step 2}. Each concentric circle $K_ \rho $ of radius $ \rho <1$ is mapped into a standard ellipse.

\noindent Proof. Let us consider a concentric ellipse $E_1\colon \frac{1}{ \rho'^2}x^2+y^2=1$ which touches the unit circle at points $(0, \pm 1)$ and touches a concentric circle $K_{ \rho '}$ with radius $ \rho  '=\frac{ \rho }{\sqrt{2- \rho ^2}}< 1$.
Since points $(0, \pm 1)$ are fixed, the image of $E_1$ is an ellipse which touches the unit circle again at points $(0, \pm1)$, therefore its axes lie on  the coordinate axes. Similarly for a concentric ellipse $E_2\colon x^2+\frac{1}{ \rho '^2}y^2=1$ which touches the unit circle at points $(\pm 1, 0)$ and touches the concentric circle $K_{ \rho '}$.  Hence, the axes of ellipses $\Phi(E_1),\Phi(E_2)$ lie on  the coordinate axes and  both ellipses touch the unit circle.
Therefore, their intersection consists of four points, symmetrically placed relative to  the $x$ and $y$ axes.

Let us note  that  ellipses $E_1$ and $E_2$  intersect at points $\frac{1}{\sqrt{2}}(\pm \rho ,\pm \rho )$  on $K_ \rho $. Therefore, the image of $K_ \rho $ must be a concentric ellipse which contains the  four  symmetrical points of intersection of $\Phi(E_1)$ and $\Phi(E_2)$. The only such ellipse  again has axes lying on  the coordinate axes.\medskip

{Step 3}.
The points inside the unit circle  belonging to   $x$-axis (respectively, $y$-axis),  are mapped into $x$-axis (respectively, $y$-axis).

\noindent Proof. The concentric circle $K_{ \rho '}$ touches the above two  ellipses  $E_1$ and $E_2$ respectively at points $(\pm \rho ',0)$ and $(0,\pm \rho ')$. These four points are mapped into intersection of the image of $E_1$ and $E_2$ with the image of $K_{ \rho '}$. The images of $E_1$, $E_2$ and $K_{\rho'}$ are  standard ellipses, moreover, the images of $E_1,E_2$ also touch the unit circle at points $(0,\pm1)$ and $(\pm1,0)$, respectively. Hence, $\Phi(E_1)$ and $\Phi(K_{\rho'})$ intersect on $x$-axis, while $\Phi(E_2)$ and $\Phi(K_{\rho'})$ intersect on $y$-axis.

{Step 4}. Let us consider again two concentric ellipses $E_1,E_2$, which touch the unit circle at symmetrically placed points $(\cos\phi,\pm \sin\phi)$ relative to $x$-axis and touch a given concentric circle $K_{\rho }$ of radius $\rho= 1/\sqrt{r}$.
These two ellipses are induced by the matrices
$$
R_{\pm\phi }\left(
\begin{smallmatrix}
 1 & 0 \\
 0 &  r  \\
\end{smallmatrix}\right)R_{\pm\phi }^T=\left(\begin{matrix}
(r-1)\sin^2{(\pm \phi)}+1 & -(r-1)\sin{(\pm \phi)}\cos{(\pm \phi)} \\
-(r-1)\sin{(\pm \phi)}\cos{(\pm \phi)} & (r-1)\cos^2{(\pm \phi)}+1 \\
\end{matrix}\right).$$
The points $\pm(\cos\phi,\pm\sin\phi)$ of intersection of ${\bf K}$ with $E_1$ and $E_2$, respectively, are mapped by $\Phi$ into points
$$\pm\frac{(\mu \cos \phi ,\pm \lambda \sin \phi )}{\sqrt{\lambda^2 \sin ^2\phi +\mu^2 \cos ^2\phi }}$$
whose argument of polar coordinates  satisfies $\tan \psi=\pm \frac{\lambda \sin \phi }{\mu \cos \phi }=\pm \frac{\lambda }{\mu}\tan \phi$, that is, $$\psi =\pm \arctan\left(\frac{\lambda }{\mu}\tan \phi \right)=\pm\arctan(\gamma\tan\phi)$$ where $\gamma:=\lambda/\mu>0$.\medskip

    {Step 5}. Since $\Phi(E_1)$ and $\Phi(E_2)$ are at angle $\pm \psi$, respectively, relative to $x$-axis and $\tan{\psi}=\pm \gamma  \tan{\phi}$, it is sufficient to replace $\tan{\phi}$ in \eqref{Rmatrixwithr} with $\gamma \tan{\phi}$ to conclude that
    $\Phi(E_1)$ and $\Phi(E_2)$ are given by the matrices
$$\frac{1}{(a-1)+(b-1)\gamma^2\tan^2{\phi}}
    \left(\begin{matrix}
    (a-1)+a(b-1)\gamma^2\tan^2{\phi} & \mp (a-1)(b-1)\gamma\tan{\phi} \\
    \mp (a-1)(b-1)\gamma\tan{\phi}  & (a-1)b+(b-1)\gamma^2\tan^2{\phi} \\
    \end{matrix}\right).$$

{Step 6}. Let us now consider ellipses $E_3,E_4$ which touch ${\bf K}$ at points $(\sin\phi,\pm\cos\phi)$ and again touch $K_{\rho }$. By repeating the previous arguments (with $\pm\phi$ replaced by $\pm\phi+\frac{\pi}{2}$) one gets that $\psi$ changes to $\psi=\pm\arctan\frac{\lambda \cot \phi }{\mu}=\pm\arctan(\gamma \cot \phi)$ and therefore
it remains to replace $\tan{\phi}$ in \eqref{Rmatrixwithr} with $\gamma \cot{\phi}$ to get that
$\Phi(E_1)$ and $\Phi(E_2)$ are induced by the matrices
$$\frac{1}{(a-1)+(b-1)\gamma^2\cot^2{\phi}}
\left(\begin{matrix}
(a-1)+a(b-1)\gamma^2\cot^2{\phi} & \mp (a-1)(b-1)\gamma\cot{\phi} \\
\mp (a-1)(b-1)\gamma\cot{\phi}  & (a-1)b+(b-1)\gamma^2\cot^2{\phi} \\
\end{matrix}\right).$$

{Step 7}. The original ellipses intersect at vertical points
$$E_1\cap E_2\ni\left(0,\pm \frac{1}{\sqrt{r \cos ^2\phi +\sin ^2\phi}}
\right).$$

 \noindent  Proof. To see this, one merely inserts $x=0$ and $y=\pm \frac{1}{\sqrt{r \cos ^2\phi +\sin ^2\phi}}$ into the equations
   $$(x,y)R_{\phi}\left(\begin{smallmatrix}
                         1 & 0\\
                         0 & r
                        \end{smallmatrix}\right)R_{\phi}^T
(x,y)^T=1=(x,y)R_{-\phi}\left(\begin{smallmatrix}
                         1 & 0\\
                         0 & r
                        \end{smallmatrix}\right)R_{-\phi}^T(x,y)^T,$$
                        and notes that being quadratic equations, they have no additional solutions with $x=0$. If we want that the vertical intersection of  $E_1$ and $E_2$  lies on a concentric circle~$K_{\rho_0}$ with $\rho_0<1$ kept fixed, one computes that the angle satisfies
                        \begin{equation}\label{eq:phi}
                         \phi=\frac{1}{2} \arccos \left(\frac{2-\rho _0^2 (r+1)}{\rho _0^2 (r-1)}\right).
                        \end{equation}

%
%

{Step 8}. Since we know that $\Phi$ maps points on $y$-axis again into points on $y$-axis, and analogously for $x$-axis, we see that the above vertical point is mapped to
   the vertical point of the intersection of $\Phi(E_1)$ and $\Phi(E_2)$,  that is to the point
$$
\left(0,\, \pm\frac{\sqrt{(a-1)+(b-1) \gamma^2
   \tan ^2\phi }}{\sqrt{(a-1) b+(b-1) \gamma^2 \tan ^2\phi}}\right).$$\medskip

{Step 9}. Likewise,  the horizontal point of intersection $E_3\cap E_4\ni \left(\pm \frac{1}{\sqrt{r \cos ^2\phi +\sin ^2\phi}},0\right)$  is mapped to the horizontal point of intersection of  $\Phi(E_3)$ and $\Phi(E_4)$, that is, to  the point
$$\left(\pm\frac{\sqrt{(a-1)+(b-1) \gamma^2
   \cot ^2\phi }}{\sqrt{(a-1)+a (b-1) \gamma^2 \cot ^2\phi }}, \, 0\right).$$

 We now interrupt the narrative with the following technical lemma which the remaining steps rely upon.

       \begin{lemma}\label{lem:ellipse}
 Let $\gamma>0$ and let $
 \widehat{E_0}\colon a_0x^2+b_0y^2=1$ be an ellipse inside ${\bf K}$. For every given angle $\phi\in(0,\frac{\pi}{2})$ consider the uniquely given four ellipses $\widehat{E_1},\widehat{E_2},\widehat{E_3},\widehat{E_4}$, contained inside  ${\bf K}$, touching ${\bf K}$, and with $\widehat{E_1},\widehat{E_2}$ having longer axis  symmetrically at angles $\pm \arctan(\gamma\tan\phi)$  and $\widehat{E_3},\widehat{E_4}$ having longer axis   symmetrically at angles $\pm\arctan(\gamma\cot\phi)$, and  such that $\widehat{E_1}\cap \widehat{E_2}$  contains the intersection of $\widehat{E_0}$ with $y$-axis and $\widehat{E_3}\cap \widehat{E_4}$ contains the intersection of $\widehat{E_0}$ with $x$-axis.

 If for each $\phi\in(0,\frac{\pi}{2})$ there exists a standard ellipse, contained inside $\widehat{E_1},\widehat{E_2},\widehat{E_3},\widehat{E_4}$ and touching all four of them, then
 $\frac{a_0-1}{b_0-1}=\gamma^2$. 
\end{lemma} \label{lem:parameters-of-ellipse}
 \begin{figure}[h!]
 \centering
   \includegraphics[width=3cm]{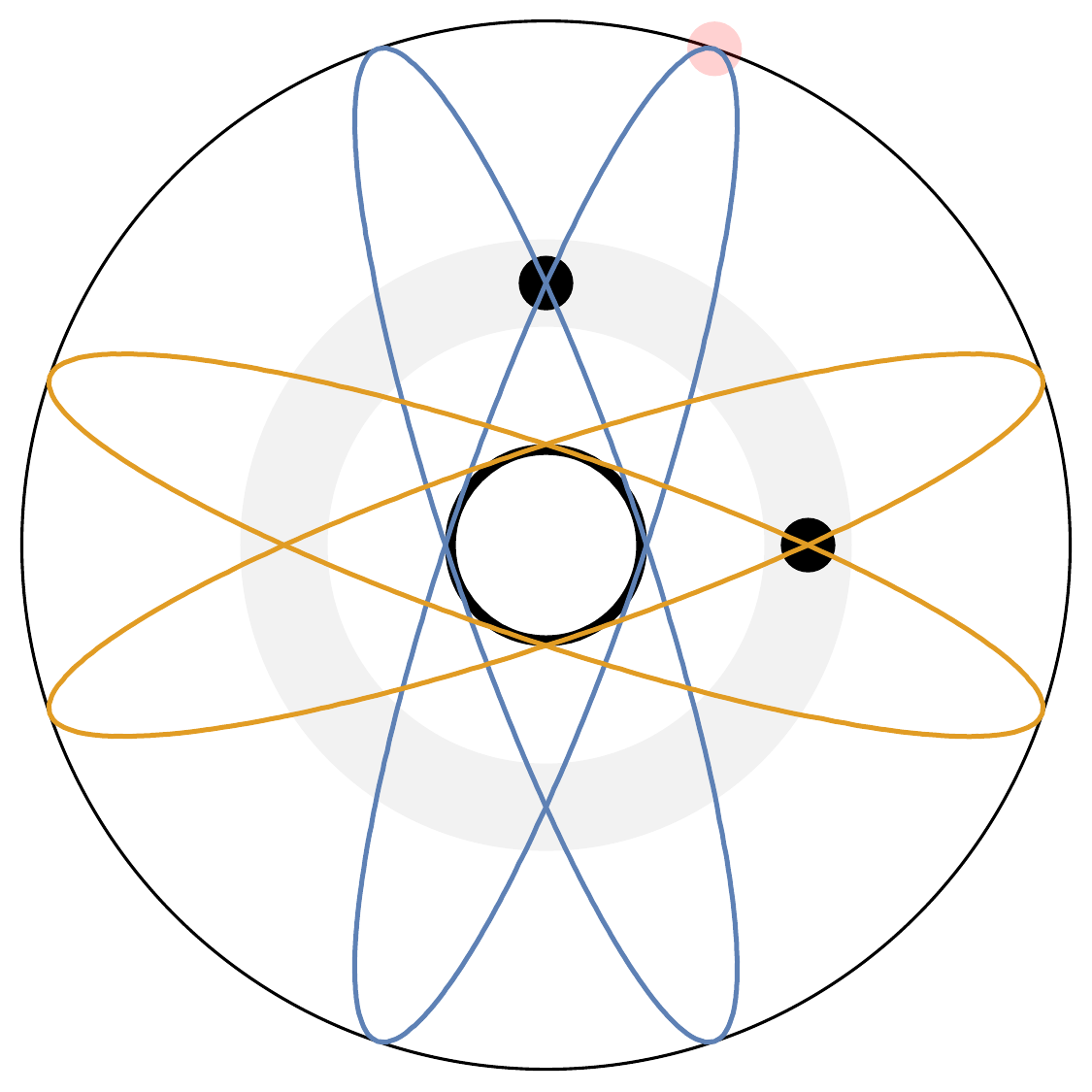}\hspace{2cm}\includegraphics[width=3cm]{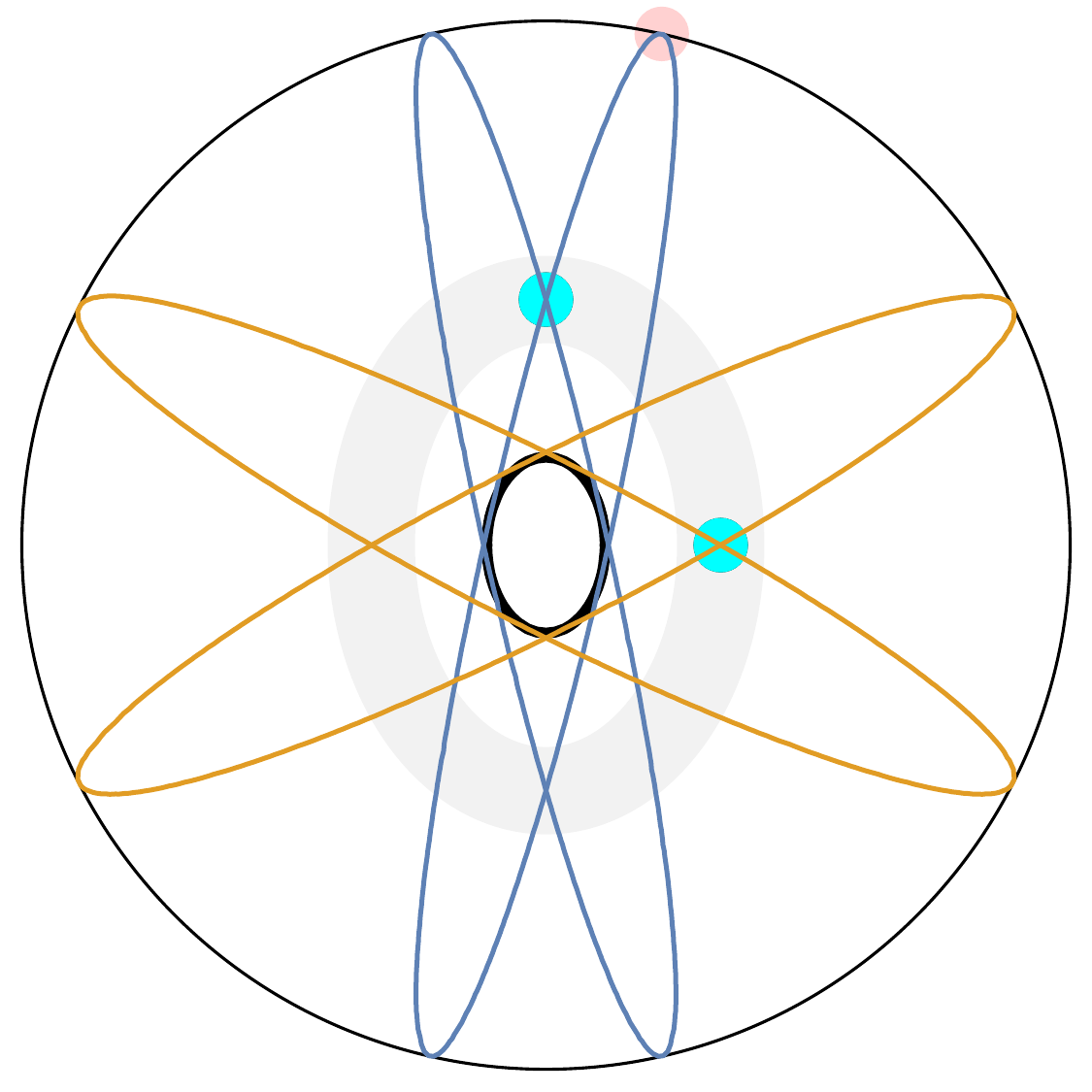}
   \caption{Four ellipses with axes symmetrically at an angle $\phi$ relative to $x$ and $y$, respectively,  touching the same inner circle as well as the unit circle, and intersecting vertically/horizontally at the same concentric circle drawn in thick gray. Their images on the right picture must intersect at the image of the gray thick circle, i.e., at the gray thick ellipse and must touch the image of the inner circle, i.e., an inner ellipse. }
 \end{figure}

 \begin{proof}
 Let the touching ellipse $\widehat{E}$ (for an angle $\phi$)  have an equation $\widehat{E}\colon ax^2+by^2=1$. Note that   $a=a(\phi)>1$ and $b=b(\phi)> 1$. Namely, if,  for some $\phi$, $a(\phi)=1$, then  the standard ellipse $\widehat{E}$ would touch ${\bf K}$ and hence would either have to intersect the slanted ellipses $\widehat{E}_i$ in four points or (if also $b(\phi)=1$) would equal ${\bf K}$ and hence would not lie inside $\widehat{E}_i$, a contradiction; similarly $b(\phi)=1$ is impossible. Set $$u=u(\phi)=a-1>0, \quad v=v(\phi)=b-1>0,\quad  t=\tan{\phi}.$$ As shown above, the four ellipses $\widehat{E_1},\widehat{E_2},\widehat{E_3},\widehat{E_4}$,  touching $\widehat{E}$, and touching the unit circle at symmetrically placed points with angles $\pm\arctan(\gamma\tan\phi)$ and $\pm\arctan(\gamma\cot\phi)$, respectively, intersect at vertical points
 $$\widehat{E_1}\cap\widehat{E_2}\cap \{(x=0)\} = \left(0,\pm\frac{\sqrt{(a-1)+(b-1) \gamma^2 \tan ^2\phi}}{\sqrt{(a-1) b+(b-1) \gamma^2 \tan
   ^2\phi}}\right)=\left(0,\pm\frac{\sqrt{u+v \gamma^2 t ^2}}{\sqrt{u (v+1)+ v\gamma^2 t  ^2}}\right)$$
   and at horizontal points
   $$\widehat{E_3}\cap\widehat{E_4}\cap \{(y=0)\} = \left(\pm\frac{\sqrt{(a-1)+(b-1) \gamma^2 \cot
   ^2\phi}}{\sqrt{{(a-1)+a (b-1) \gamma^2 \cot ^2\phi}}},0\right)=\left(\pm\frac{\sqrt{u+v \gamma^2 /t^2}}{\sqrt{{u+(u+1)v \gamma^2/t^2}}},0\right).$$
These points must belong to $\widehat{E_0}$, so that
$$\frac{\sqrt{u+v \gamma^2
   /t^2}}{\sqrt{{u+(u+1) v \gamma^2 /t^2}}}=\frac{1}{\sqrt{a_0}},\qquad \frac{\sqrt{u+v \gamma^2 t^2}}{\sqrt{u (v+1)+v \gamma^2 t^2}}=\frac{1}{\sqrt{b_0}}.$$
Then, after squaring and inverting, these two equations simplify to
$$a_0-1=\frac{uv\gamma^2}{ut^2+\gamma^2v}, \qquad b_0-1=\frac{uv}{\gamma^2vt^2+u}.$$
The first equation implies
$$u=u_t=\frac{(a_0-1)\gamma^2v}{\gamma^2v-(a_0-1)t^2},$$
which we insert in the second equation to get
$$v=v_t=\frac{\left(a_0-1\right) \left(1-t^4\right)}{\frac{a_0-1}{b_0-1}-\gamma ^2 t^2}.$$
Since $v=v_t=b-1> 0$ for each $t>0$ and the numerator of $v_t$ changes sign at $t=1$, the denominator must do the same, so at $t=1$ the denominator vanishes. This gives
$$\frac{a_0-1}{b_0-1}=\gamma^2$$
and the result follows. Let us remark for the future use  that with this  $\gamma$ we have
$v=(b_0-1)(1+t^2)$ and $u=(a_0-1)(1+t^2)$,
which implies that the parameters of a touching ellipse (with $t=\tan\phi$) are
\begin{equation}\label{eq:ab}
a=a_0(1+t^2)-t^2, \qquad b=b_0(1+t^2)-t^2=\frac{1}{\gamma^2}(a_0-1)(1+t^2)+1.\qedhere
\end{equation}
\end{proof}
We  now continue with the steps.\medskip


{Step 10}.  By Step~2, $\Phi$ maps the concentric circle $K_{1/2}$ of radius $\rho_0=1/2$ into a standard ellipse
$$\widehat{E}_0=\Phi(K_{1/2})\colon a_0 x^2+b_0y^2=1.$$
Let $\rho=1/\sqrt{r}\le1/2$ be the radius of a concentric circle $K_{\rho}$. By Step~7, if we place  ellipses $E_1,E_2,E_3,E_4$, each having a long axis of length $1$ and a short axis of length $\rho$,  symmetrically (relative to $x$-axis  and $y$-axis)  at an angle $\pm\phi$ (respectively, $\pi/2\pm\phi$ for $E_3,E_4$), with  $\phi$  given  by~\eqref{eq:phi}, then they will intersect vertically and   horizontally, respectively, on a fixed circle $K_{\rho_0}$, and moreover, all of them will touch ${\bf K}$ and will have  $K_{\rho}$ as a common incircle. The point $(\cos\phi,\sin\phi)$  where $E_1$ touches ${\bf K}$ is mapped to a point
$\Phi((\cos\phi,\sin\phi))=(\mu\cos\phi,\lambda \sin\phi)$ where  $\widehat{E}_1=\Phi(E_1)$ touches  $\Phi({\bf K})={\bf K}$.

By Lemma~\ref{lem:Phi(interior)=interior}, $\widehat{E}_1$ is contained in ${\bf K}$, and hence its longer axis is of length $1$ at the angle $\psi=\arctan(\tfrac{\lambda\sin\phi}{\mu\cos\phi})=\arctan(\gamma\tan\phi)$. Likewise we see that  $\widehat{E}_i=\Phi(E_i)$, $i=1,2,3,4$ satisfy all the assumptions on axis of Lemma~\ref{lem:ellipse}. To see that $\widehat{E}_1\cap \widehat{E}_2$  intersects $\widehat{E}_0=\Phi(K_{1/2})$ on $y$-axis we use Step~3 (and likewise that $\widehat{E}_3\cap \widehat{E}_4$  intersects $\widehat{E}_0$ on $x$-axis)  while   $\widehat{E}=\Phi(K_\rho)$ is their touching ellipse. Then, Lemma~\ref{lem:ellipse} implies that the parameters of $\widehat{E}_0=\Phi(K_{1/2})$  must satisfy  $(a_0-1)/(b_0-1)=\gamma^2$.

To see from here where a concentric circle of radius $\rho=1/\sqrt{r}\le\rho_0=1/2$ is mapped into, we insert into~\eqref{eq:ab} back again $t=\tan\phi$    (again, using~\eqref{eq:phi} with $\rho_0=1/2$ for $\phi$) and simplify using $\tan \left(\frac{a}{2}\right)=\sqrt{\frac{1-\cos (a)}{1+\cos (a)}}$
to get that the circle $K_{\rho}$ of radii $\rho=1/\sqrt{r}<1/2$ is mapped into an ellipse $ax^2+by^2=1$ where
\begin{equation}\label{eq:new-ab}
a= \frac{(a_0-1) (r-1)}{3} +1,\qquad b= \frac{\left(a_0-1\right) (r-1)}{3 \gamma ^2}+1.
\end{equation}

{Step 11}. Let $z=x+iy\in\CC=\RR^2$ be a point inside $K_{1/2}$. Then,
\begin{equation}\label{eq:Phi(x+iy)}
\Phi(x+i y)=\pm\frac{\sqrt{3} (x+i \gamma  y)}{\sqrt{\left(a_0-1\right) \left(1-x^2-y^2\right)+3 \left(x^2+\gamma ^2 y^2\right)}}.
\end{equation}

\noindent Proof. Let us draw an ellipse $E_{z}$  with axis on  a ray  $[0,z]$ and passing through $z$  at its extreme.  Now let us draw an incircle $K_{\mathrm{inc}}$ to this ellipse and an ellipse $E_{\mathrm{out}}$ with the same axis, the same incircle, but also touching  the unit circle. Moreover, let us draw a concentric circle $K_{|z|}$ passing through $z$. The ellipse $E_z$ is mapped into an ellipse touching $\Phi(E_{\mathrm{out}})$ and  $\Phi(K_{\mathrm{inc}})$  at the same point.
\begin{figure}[h!]
  \centering
  \includegraphics[width=3cm]{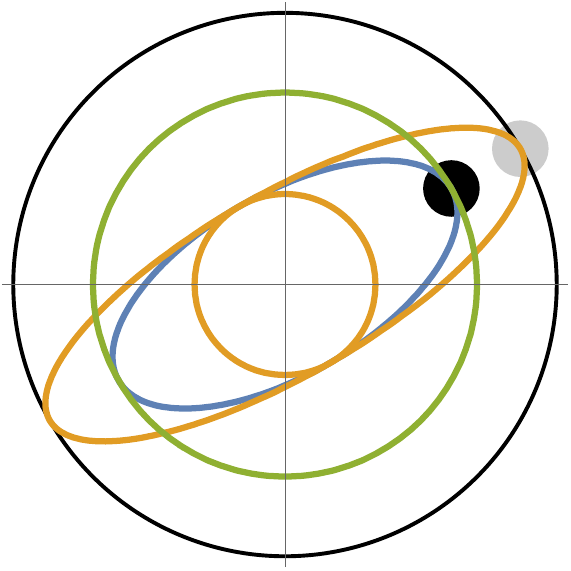}\hspace{3cm}\includegraphics[width=3cm]{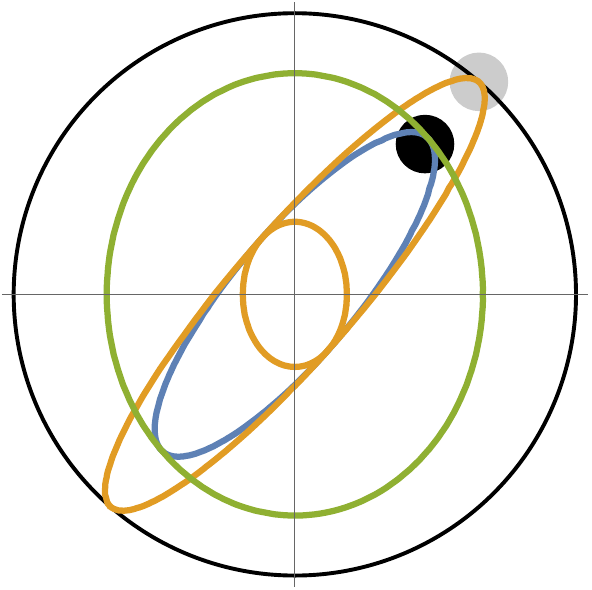}
  \caption{Towards the proof of Step 11. The inner-most circle can be arbitrarily small, so arbitrarily narrow ellipse  touches a unit circle and has axis on a ray passing through a given point. On the right are their $\Phi$-images.}\label{fig2}
\end{figure}

Now, we are free to  take the ellipse $E_z$ inside ${\bf K}$ as long as it touches the point $z$ at one of its axis. Its second axis can be arbitrarily small, in fact there exists a sequence of ellipses  $E_z$  converging to a line segment
$[-z,z]$;
these ellipses have arbitrarily small inscribed circle so the  corresponding sequence of outer   ellipses $E_{\mathrm{out}}$  touching the unit circle along the same axis and touching the inscribed circles, are also converging to a line segment.  The  radius $\rho_{\text{inc}}=1/\sqrt{r_{\text{inc}}}$ of inner inscribed circles converge to zero.  They are mapped into ellipses  with parameters described in \eqref{eq:new-ab} and clearly with $\rho_{\text{inc}}\to0$, i.e., with $r=r_{\text{inc}}\to\infty$, both parameters go to $\infty$  so that the ellipse $\Phi(K_{\rho })$  shrinks to a point. This means that $\Phi(E_{\text{out}})$
which touches the shrinking ellipses $\Phi(K_{\rho })$, must also shrink to a line segment.
 Thus,   $\Phi(E_z)$, which always lies inside  $\Phi(E_{\text{out}})$, also shrinks to a line segment and it always  touches $\Phi(K_{|z|})$.

 Notice that this line segment contains a point $\Phi(E_{\text{out}}\cap {\bf K})=\Phi(\tfrac{x+iy}{|x+iy|})=\frac{\mu x+i \lambda y}{|\mu x+i \lambda y|}$. This shows that $\Phi(z)$ is the intersection of $\Phi(K_{|z|)})$ and a line segment through $\pm (\mu x+i\lambda y) $.

From here, and using that $\Phi(K_{|z|})$ is an ellipse of the form $aX^2+bY^2=1$ with parameters $a,b$  given in \eqref{eq:new-ab} (with $r=\frac{1}{|z|^2}=\frac{1}{x^2+y^2}$ and $\gamma=\lambda/\mu$), one computes the asserted form:
\begin{eqnarray}
  \Phi(x+i y)&=& \pm\frac{\mu  x+i \lambda  y}{\sqrt{a \mu ^2 x^2+b \lambda ^2 y^2}}=\pm\frac{x+i \gamma  y}{\sqrt{a x^2+b \gamma ^2 y^2}}\notag\\
  &=&\pm  \frac{\sqrt{3} (x+i \gamma  y)}{\sqrt{\left(4-a_0\right) x^2+y^2 \left(1-a_0+3 \gamma ^2\right)+\left(a_0-1\right)}}\label{eq:domeins-Phi}
;\qquad |x+iy|\le \tfrac{1}{2},\\
&=&\pm\frac{\sqrt{3} (x+i \gamma  y)}{\sqrt{\left(a_0-1\right) \left(1-x^2-y^2\right)+3 \left(x^2+\gamma ^2 y^2\right)}};\qquad |x+iy|\le \tfrac{1}{2}.\notag
\end{eqnarray}

{Step 12}. The function on the right-hand side of~\eqref{eq:Phi(x+iy)}, which we denote by $\widehat{\Phi}$, is clearly defined for every $z=x+iy\in\CC$ inside the unit circle ${\bf K}$, and it coincides with $
\Phi$ inside  $K_{1/2}$. We claim that this function maps quadrics into quadrics.

\noindent Proof. To see this note that a general quadric takes the form  $a x^2 + b x y + c y^2 = 1$ for some parameters $a,b,c$.
One can solve it for $y$ and insert the $\widehat{\Phi}$-image of parametrization $(x,y(x))$   into a polynomial $AX^2+B XY+C Y^2$. By letting
\begin{equation}\label{eq:ABC}
 \begin{aligned}
 A&= \frac{1}{3} \left(a \left(a_0-1\right)-a_0+4\right),\\
B&= \frac{\left(a_0-1\right) b}{3 \gamma }, \\
C&=\frac{\left(a_0-1\right) (c-1)}{3 \gamma ^2}+1
\end{aligned}
\end{equation}
 the polynomial simplifies into a constant $1$, as claimed.\medskip

{Step 13}.  The function $\hat{\Phi}$ can identify two points only if they belong to the same line through the origin. Moreover,
\begin{equation}\label{eq:surject}
a_0\le 4\qquad \hbox{and}\qquad a_0\le 1+3\gamma^2.
\end{equation}
\noindent Proof. The first claim is a consequence of the fact that
$$\frac{\mathrm{Re}(\hat{\Phi}(x+iy))}{\mathrm{Im}(\hat{\Phi}(x+iy))}=\frac{x}{\gamma y}.$$
In particular, if a quadric $Q$ intersects the interior or $K_{1/2}$, then $\hat{\Phi}(Q\cap \mathrm{int} K_{1/2})$ contains infinitely many points, so by Step~12 belongs to a unique quadric.

To prove the second statement, let us consider first a point $T_\alpha=(\alpha,0)$ with $\alpha>1$. Each ellipse of the family 
$E_c\colon x^2/\alpha^2+c y^2=1$ ($c>4$)  contains $T_\alpha$ and passes through the interior of $K_{1/2}$. Hence, the ellipse $\Phi(E_c)$ is uniquely determined by its subset $\Phi(E_c\cap \mathrm{int}K_{1/2})=\hat{\Phi}(E_c\cap \mathrm{int}K_{1/2})$. In fact, it equals the quadric determined by coefficients \eqref{eq:ABC}, that is, the quadric
$$\frac{\left(\frac{1}{\alpha^2} \left(a_0-1\right)-a_0+4\right)}{3} x^2+\frac{\left(a_0-1\right) (c-1)+3\gamma^2}{3 \gamma ^2}y^2=1,$$
which contains a point $\Phi(\alpha,0)$ for each $c>1$. Letting  $c\to\infty$  we see that $\Phi(\alpha,0)=(x_\alpha,0)$  where $x_\alpha=\pm \frac{\sqrt{3}}{\sqrt{\frac{(a_0-1)}{\alpha^2} +4-a_0}}$. This is clearly defined for each $\alpha>1$ and letting $\alpha \to\infty$ we see that $4-a_0\ge 0$ (in order that the denominator in $x_\alpha$ is never imaginary).

Similarly, the point $T'_\beta=(0,\beta)$ (with $\beta>1$) belongs to each ellipse of the family  $E'_\beta\colon a x^2+y^2/\beta^2=1$, here with $a>4$. Arguing as before, $\Phi(E'_\beta)$ is an ellipse with the equation
$$\frac{\left(a \left(a_0-1\right)-a_0+4\right)}{3}  x^2+\frac{\left(a_0-1\right) (\frac{1}{\beta^2}-1)+3\gamma^2}{3 \gamma ^2}y^2=1$$
and it always contains $\Phi(0,\beta)$. For $a\to\infty$ we get $\Phi(0,\beta)= (0,y_\beta)$ with $y_\beta=\frac{\sqrt{3\gamma^2}}{\sqrt{\left(a_0-1\right) (\frac{1}{\beta^2}-1)+3\gamma^2}}$. This is defined for each $\beta>1$ and with $\beta\to\infty$ we must have $-(a_0-1)+3\gamma^2\ge 0$ in order the the denominator is never imaginary.

 {Step 14}. It follows from Step 13 and  \eqref{eq:domeins-Phi} that $\widehat{\Phi}(z)$ is defined for every $z\in\CC$.
 We claim that $\Phi(z)=\widehat{\Phi}(z)$ for every $|z|\ge 0$ and not just for $|z|\le 1/2$.

To see this, let us consider two distinct ellipses $E_z^{1}$ and $E_z^{2}$ both of which contain a point~$z$ and points from $\mathrm{int} K_{1/2}$. Then,
\begin{equation}\label{eq:Phi(z)}
                      \Phi(z)\in \Phi(E_z^{1})\cap \Phi(E_z^{2}).                                                                                                                                                                           \end{equation}
                  By the assumptions, $\Phi(E_z^{i})$ is an ellipse while $\widehat{\Phi}(E_z^{i})$ is a quadric, which coincides with the ellipse $\Phi(E_z^{i})$ on infinitely  many points (the~$\hat{\Phi}$-images of $E_z^{i}\cap \mathrm{int} K_{1/2})$.  It follows that  $\widehat{\Phi}(E_z^{i})=\Phi(E_z^{i})$, and the claim follows  from~\eqref{eq:Phi(z)}.\medskip

{Step 15}.
Since   $\Phi=\widehat{\Phi}$ is surjective, for every $u+iv$ there exists $x+iy$ with $\hat{\Phi}(x+iy)=u+iv$ (recall that the action of $\hat{\Phi}$ is described in~\eqref{eq:domeins-Phi}). Thus, we need to solve a system of two equations with real variables $x,y$. By squaring both equations we get a system
\begin{equation*}
 \begin{aligned}
  \frac{3 x^2}{\left(4-a_0\right) x^2+y^2 \left(1-a_0+3 \gamma ^2\right)+a_0-1}&=u^2,\\
  \frac{3 \gamma ^2 y^2}{\left(4-a_0\right) x^2+y^2 \left(1-a_0+3 \gamma ^2\right)+a_0-1}&=v^2
 \end{aligned}
\end{equation*}
with solution
\begin{equation*}
 \begin{aligned}
x^2&= \frac{\left(a_0-1\right) \gamma ^2 u^2}{\left(a_0-4\right) \gamma ^2 u^2+v^2 \left(a_0-1-3 \gamma ^2\right)+3 \gamma ^2},\\
y^2&= \frac{\left(a_0-1\right) v^2}{\left(a_0-4\right) \gamma ^2 u^2+v^2 \left(a_0-1-3 \gamma ^2\right)+3 \gamma ^2}.
 \end{aligned}
\end{equation*}
Since $a_0>1$, the numerator is in both cases nonnegative, so denominator must be nonnegative as well. Taking $v=0$ and $u\to\infty$ we see that $a_0-4\ge0$, that is, $a_0\ge 4$. Taking $u=0$ and $v\to\infty$ we get $a_0-1-3\gamma^2\ge 0$. So $\Phi$ is surjective if and only if
$$a_0\ge 4\qquad \hbox{and}\qquad a_0\ge 1+3\gamma^2. $$
Combining this with \eqref{eq:surject} we deduce
$$a_0=4\quad\hbox{ and } \quad \gamma=1.$$
Thus, $T=\diag(\mu,\lambda)=\diag(\mu,\gamma\mu)$ is a scalar matrix. Moreover,  from the  equation  for $\Phi$, we see that $\Phi(x+iy)=x+iy$ is the identity.\qed

\section{Preservers of the minus partial order}

Let $A,B\in M_{n}(\mathbb{F)}$. It is known (see, e.g. \cite[page 149]%
{Legisa}) that
\begin{equation}
A\leq^{-}B\quad\text{if and only if}\quad\func{Im}B=\func{Im}A\oplus\func{Im}%
(B-A)\quad\text{if and only if}\quad RAL\leq^{-}RBL  \label{eq_minus}
\end{equation}
for any invertible $R,L\in M_{n}(\mathbb{F)}$. Let $A,B\in M_{n}(\mathbb{C)}$%
. If there exists an invertible matrix $S\in M_{n}(\mathbb{C)}$ such that

\begin{itemize}
\item[a)] $B=$ $SAS^{T}$, then we say that $A$ and $B$ are congruent;

\item[b)] $B=$ $SAS^{\ast}$, then we say that $A$ and $B$ are *congruent.
\end{itemize}

By Sylvester's law of inertia (see \cite[page 282]{Horn}) two (Hermitian)
matrices $A,B\in H_{n}(\mathbb{C)}$ are *congruent if and only if they have
the same inertia, i.e., they have the same number of positive eigenvalues
and the same number of negative eigenvalues. Two (real symmetric) matrices $%
A,B\in H_{n}(\mathbb{R)}$ are *congruent via a complex matrix if and only if
they are congruent via a real matrix \cite[page 283]{Horn}. So, Sylvester's
law for the real case states that $A,B\in H_{n}(\mathbb{R)}$ are congruent
via an invertible $S\in M_{n}(\mathbb{R)}$ (i.e., $B=$ $SAS^{T}$) if and
only if $A$ and $B$ have the same number of positive eigenvalues and the
same number of negative eigenvalues. Note that congruent (respectively,
*congruent) matrices have the same rank \cite[page 281]{Horn}.

Denote by $E_{ij}$ the $n\times n$ matrix with all entries equal to zero
except the $(i,j)$-entry which is equal to one. Let $E_{k}=E_{11}+E_{22}+%
\ldots +E_{kk}$. For $A,B\in M_{n}(\mathbb{R)}$ we will write $A<^{-}B$ when
$A\leq ^{-}B$ and $A\neq B$. We now state and prove our main result.

\begin{theorem}
Let $n\geq 3$ be an integer. Then $\Phi :H_{n}^{+}(\mathbb{R})\rightarrow
H_{n}^{+}(\mathbb{R})$ is a surjective bi-monotone map with respect to the
minus partial order $\leq ^{-}$ if and only if there exists an invertible matrix $%
S\in M_{n}(\mathbb{R)}$ such that
\begin{equation*}
\Phi (A)=SAS^{T}
\end{equation*}%
for every $A\in H_{n}^{+}(\mathbb{R})$.
\end{theorem}

\begin{proof}
Let $\Phi :H_{n}^{+}(\mathbb{R})\rightarrow H_{n}^{+}(\mathbb{R})$ be of the
form $\Phi (A)=SAS^{T}$, $A\in H_{n}^{+}(\mathbb{R})$, where $S\in M_{n}(%
\mathbb{R)}$ is an invertible matrix. Then $\Phi $ preserves by (\ref%
{eq_minus}) the order $\leq ^{-}$ in both directions and is clearly
surjective.

Conversely, let $\Phi :H_{n}^{+}(\mathbb{R})\rightarrow H_{n}^{+}(\mathbb{R}%
) $ be a surjective map that preserves the order $\leq ^{-}$ in both
directions. We will split the proof into several steps.

Step 1. $\Phi $ \textit{is bijective. }Let $\Phi (A)=\Phi (B)$ for $A,B\in
H_{n}^{+}(\mathbb{R})$. The order $\leq ^{-}$ is reflexive so $\Phi (A)\leq
^{-}\Phi (B)$ and $\Phi (B)\leq ^{-}\Phi (A)$. Since $\Phi $ preserves the
order $\leq ^{-}$ in both directions, we have $A\leq ^{-}B$ and $B\leq ^{-}A$%
. It follows that $A=B$, since $\leq ^{-}$ is antisymmetric. Thus, $\Phi $
is injective and therefore bijective.

Step 2. $\Phi (0)=0$. Note that $0\leq ^{-}A$ for every $A\in H_{n}^{+}(%
\mathbb{R})$. So, on the one hand $0\leq ^{-}\Phi (0)$ and on the other
hand, since $\Phi ^{-1}$ has the same properties as $\Phi $, $0\leq ^{-}\Phi
^{-1}(0)$ and thus $\Phi (0)\leq ^{-}0$.

Step 3. $\Phi $ \textit{preserves the rank, i.e., }$\limfunc{rank}(A)=%
\limfunc{rank}(\Phi (A))$ \textit{for every} $A\in H_{n}^{+}(\mathbb{R})$.
Let $A\in H_{n}^{+}(\mathbb{R})$ with $\limfunc{rank}(A)=k$. By Sylvester's
law of inertia there exists an invertible matrix $R\in M_{n}(\mathbb{R)}$
such that $E_{k}=RAR^{T}$. Clearly (see (\ref{eq_rank_minus})),
\begin{equation*}
0<^{-}E_{1}<^{-}E_{2}<^{-}\ldots <^{-}E_{n}=I.
\end{equation*}%
Since congruence preserves rank, we have by (\ref{eq_minus})
\begin{equation*}
0<^{-}R^{-1}E_{1}(R^{-1})^{T}<^{-}R^{-1}E_{2}(R^{-1})^{T}<^{-}\ldots
<^{-}R^{-1}E_{k}(R^{-1})^{T}<^{-}\ldots <^{-}R^{-1}E_{n}(R^{-1})^{T}.
\end{equation*}%
From $(R^{-1})^{T}=(R^{T})^{-1}$ and since $\Phi $ preserves the order $\leq
^{-}$ and it is injective, we obtain
\begin{equation}
0<^{-}\Phi (R^{-1}E_{1}(R^{T})^{-1})<^{-}\Phi
(R^{-1}E_{2}(R^{T})^{-1})<^{-}\ldots <\Phi (A)<^{-}\ldots <^{-}\Phi
(R^{-1}(R^{T})^{-1}).  \label{rank_phi_minus}
\end{equation}

Let $C,D\in M_{n}(\mathbb{R)}$ with $C<^{-}D$ and $\limfunc{rank}(C)=%
\limfunc{rank}(D)$. Then by (\ref{eq_rank_minus}), $\limfunc{rank}(D-C)=0$
and therefore $D=C$, a contradiction. So, if $C<^{-}D$, then $\limfunc{rank}%
(C)<\limfunc{rank}(D)$.

Every succeeding matrix in (\ref{rank_phi_minus}) has the rank that is
strictly greater then its predecessor. Since $\limfunc{rank}\Phi
(R^{-1}(R^{T})^{-1})\leq n$, it follows that $\limfunc{rank}\Phi
(R^{-1}(R^{T})^{-1})=n$ and therefore $\limfunc{rank}(\Phi (A))=k$.

Step 4. \textit{We may without loss of generality assume that} $\Phi (I)=I$.
By the previous step, $\Phi (I)=B$ where $B\in H_{n}^{+}(\mathbb{R})$ is an
invertible (positive definite) matrix. It follows that there exists a
positive definite matrix $\sqrt{B}\in H_{n}^{+}(\mathbb{R})$ such that $\Phi
(I)=\sqrt{B}\sqrt{B}$. Let $\Psi :$ $H_{n}^{+}(\mathbb{R})\rightarrow
H_{n}^{+}(\mathbb{R})$ be defined by
\begin{equation*}
\Psi (A)=\left( \sqrt{B}\right) ^{-1}\Phi (A)\left( \sqrt{B}\right) ^{-1}.
\end{equation*}%
Then $\Psi $ is a bijective map that preserves the order $\leq ^{-}$ in both
directions. Also, $\Psi (I)=I$. We will thus from now on assume that
\begin{equation*}
\Phi (I)=I.
\end{equation*}

Step 5. \textit{There exists a bijective, linear map }$T:\mathbb{R}%
^{n}\rightarrow \mathbb{R}^{n}$ \textit{such that for every }$P\in P_{n}(%
\mathbb{R})$ \textit{the matrix }$\Phi (P)$ \textit{is the projector onto }$T(\func{Im}P)$\textit{, i.e., }%
\begin{equation*}
\mathit{\ }\Phi (P)=P_{T(\func{Im}P)}.
\end{equation*}

Let $P\in M_{n}(\mathbb{R})$ be an idempotent matrix, i.e., $P^{2}=P$. Then $%
\mathbb{R}^{n}=\func{Im}P\oplus \limfunc{Ker}P=\func{Im}P\oplus \mathrm{%
\func{Im}}(I-P)$ and therefore by (\ref{eq_minus}), $P\leq ^{-}I$. Moreover,
if $Q\in M_{n}(\mathbb{R)}$ is an idempotent matrix and if $A\leq ^{-}Q$ for
$A\in M_{n}(\mathbb{R)}$, then by, e.g. \cite[Lemma 2.9]%
{MarovtRakicDjodjevic}, $A^{2}=A$. Thus for $P\in M_{n}(\mathbb{R})$ we have
\begin{equation*}
P\leq ^{-}I\quad \text{if and only if}\quad P^{2}=P.
\end{equation*}%
Let now $P\in $ $P_{n}(\mathbb{R})$, i.e., $P$ is a symmetric and idempotent
matrix. It follows that $P\leq ^{-}I$ and therefore $\Phi (P)\leq ^{-}\Phi
(I)=I$. So, $\Phi (P)$ is an idempotent matrix and by the definition of the
map $\Phi $ also symmetric, i.e., $\Phi (P)\in $ $P_{n}(\mathbb{R})$. Since $%
\Phi ^{-1}$ has the same properties as $\Phi $, we may conclude that
\begin{equation*}
P\in P_{n}(\mathbb{R})\quad \text{if and only if}\quad \Phi (P)\in P_{n}(%
\mathbb{R}),
\end{equation*}%
i.e., $\Phi $ preserves the set of all projectors.
Recall that we may identify subspaces of $\mathbb{R}^{n}$ with elements of $%
P_{n}(\mathbb{R})$. Let $\mathcal{C}(\mathbb{R}^{n})$ be the lattice of all
subspaces of $\mathbb{R}^{n}$. It follows that the map $\Phi $ induces a
lattice automorphisms, i.e., a bijective map $\tau :\mathcal{C}(\mathbb{R}%
^{n})\rightarrow \mathcal{C}(\mathbb{R}^{n})$ such that
\begin{equation*}
M\subseteq N\quad \text{if and only if}\quad \tau (M)\subseteq \tau (N)
\end{equation*}%
for all $M,N\in \mathcal{C}(\mathbb{R}^{n})$. In \cite[page 246]{Mackey}
(see also \cite[pages 820 and 823]{FillmoreLongstaff} or \cite[page 82]%
{Pankov}) Mackey proved that for $n\geq 3$ every such a map is induced by an
invertible linear operator, i.e., there exists an invertible linear operator
$T:$ $\mathbb{R}^{n}\rightarrow \mathbb{R}^{n}$ such that $\tau (M)=T(M)$
for every $M\in $ $\mathcal{C}(\mathbb{R}^{n})$. For the map $\Phi $ it
follows that
\begin{equation}
\Phi (P)=P_{T(\func{Im}P)}  \label{eq_projectors}
\end{equation}%
for every $P=P_{\func{Im}P}\in P_{n}(\mathbb{R})$.

Step 6. \textit{There exist orthogonal matrices $U,V \in M_{n}(\mathbb{R)}$, such that $U^{T} \Phi (V^{T} A V) U = A$}
\textit{for every rank-one} $A\in H_{n}^{+}(\mathbb{R})$.
Note that for $T$ from Step 5 there exists a singular value decomposition $T= UDV$, with $U, V$ orthogonal
and $D = \mathrm{diag\, } ( \lambda_1, \lambda_2, \ldots, \lambda_n )$ being
a diagonal matrix with $\lambda_1 = \max \lambda_i$ and $\lambda_2 = \min \lambda_i$. By replacing the map $\Phi$ with the map $X\mapsto U^{T} \Phi (V^{T} X V) U$, $X \in M_{n}(\mathbb{R)}$, we can assume that $\Phi(I) = I$ and $\Phi(x \otimes
x^{*}/\|x\|^2) = (Dx) \otimes (Dx)^{*} /\| Dx\|^2$ for every nonzero $x \in \mathbb{R}^n$. In particular, $\Phi(E_{ii}) = E_{ii}$.
By \eqref{eq_projectors} and since $T$ is now diagonal we have $\Phi(I_2 \oplus 0_{n-2}) = I_2 \oplus 0_{n-2}$.

Consider a rank-two matrix $A \in H_{2}^{+}(\mathbb{R}) \oplus 0_{n-2}$. Note that the set $\Xi_A = \{x \otimes x^{*} \leq ^{-} A\}$ is mapped by $\Phi$ onto
$\Xi_{\Phi(A)}$ since $\Phi$ is bi-monotone. By Lemma~\ref{lem:2.2}, the set $\Xi_A$ is in a bijective correspondence with the set of unordered pairs of antipodal points on a concentric planar ellipse $\mathfrak{E}_{A}=\left\{ x\in \func{Im}A:x^{\ast }A^{\dagger }x=1\right\} \subseteq \RR^2 \oplus 0_{n-2}$, see \eqref{eq:ellipse}.
By Step 3, $\Phi(A)$ is also a rank-two matrix. Hence $\Phi$ induces a map, which we still denote by $\Phi$, that maps concentric ellipses from $\RR^2$ into concentric planar ellipses. By Lemma~\ref{lem:3.1}, they all lie in the same plane, which is equal to $\Image \Phi(I_2 \oplus 0_{n-2}) = \Image I_2 \oplus 0_{n-2} = \RR^2 \oplus 0_{n-2}$.
Therefore the induced $\Phi$ bijectively maps the points from $\RR^2$ onto $\RR^2$, modulo identifying pairs of points symmetric to the origin, and maps concentric ellipses onto concentric ellipses. Notice also that $\mathfrak{E}_{I_2 \oplus 0_{n-2}}$ is a unit circle by Lemma~\ref{lem:2.2}. Since rank-one projectors which are below $I_2 \oplus 0_{n-2}$ are identified with (antipodal) points on the unit circle, we see that $\Phi(\mathfrak{E}_{I_2 \oplus 0_{n-2}} ) = \mathfrak{E}_{I_2 \oplus 0_{n-2}}$. Also, a point $(c,s)$ on the unit circle corresponds to a projector $(c,s,0,\ldots,0)\otimes(c,s,0,\ldots,0)^{\ast}$ which is mapped to a projector
\begin{equation*}
\begin{aligned}
\Phi((c,s,0,\ldots,0)\otimes(c,s,0,\ldots,0)^{\ast})&=\frac{D(ce_1+se_2)\otimes(D(ce_1+se_2))^{\ast}}{\|D(ce_1+se_2)\|^2}\\
&=\frac{(\lambda_1 ce_1+\lambda_2 se_2)\otimes(\lambda_1 ce_1+\lambda_2 se_2)^{\ast}}{\|\lambda_1 ce_1+\lambda_2 se_2\|^2}
\end{aligned}
\end{equation*}
that corresponds to a point $\frac{(\lambda_1 c,\lambda_2 s)}{\|(\lambda_1 c,\lambda_2 s)\|}$ also on the unit circle. Hence induced $\Phi$ satisfies all the assumptions of
Section 3 and therefore $\lambda_1=\lambda_2$ and thus $D$ is a scalar matrix. So, $\Phi$ fixes all projectors. Moreover, by Section 3 the induced map $\Phi$ is an identity on $\RR^2$ so the map $\Phi$ fixes all rank-one matrices which are below $I_2 \oplus 0_{n-2}$. Now we consider  an arbitrary rank-one matrix $R\in H_{n}^{+}(\mathbb{R})$ and choose any rank-two matrix $B\in H_{n}^{+}(\mathbb{R})$ that majorizes $R$ with respect to the minus partial order. There exists an orthogonal matrix $W\in M_{n}(\mathbb{R})$ such that $\Image (WBW^{T})=\RR^2\oplus 0_{n-2}$.  Note that $\Image B =\Image P_{\Image{B}}$ where $P_{\Image{B}}$ is the projector onto the image of $B$, so again by Lemma~\ref{lem:3.1},  $\Image\Phi(B)=\Image\Phi(P_{\Image{B}})$. Since we know that $\Phi$ fixes all the projectors, we see that $\Image\Phi(B)=\Image B$. By repeating the above arguments on the map $X\mapsto W\Phi(W^{T}XW)W^{T}$ which also fixes all projectors and maps $\RR^2\oplus 0_{n-2}$ to $\RR^2\oplus 0_{n-2}$, we see that $\Phi$ also fixes  the rank-one matrix $R$.


Step 7. \textit{Conclusion of the proof. } Let $A\in H_{n}^{+}(\mathbb{R})$
be of rank-one. By Step 6, $\Phi (A)=A$. It
follows immediately by Remark \ref{rem:geometric} that then $\Phi (A)=A$ for
every $A\in H_{n}^{+}(\mathbb{R})$. Taking into account all the assumptions
we may conclude that $\Phi :H_{n}^{+}(\mathbb{R})\rightarrow H_{n}^{+}(\mathbb{R})$ is a surjective bi-monotone map with respect to the minus partial order $\leq ^{-}$ if and only if there exists an invertible matrix $S\in M_{n}(\mathbb{R)}$ such that
\begin{equation*}
\Phi (A)=SAS^{T}
\end{equation*}for every $A\in H_{n}^{+}(\mathbb{R})$.
\end{proof}

\end{document}